\documentclass[english,12pt]{amsart}
\title{Transience of Edge-Reinforced Random Walk}

\usepackage{amsmath,amsfonts,bbm,amsthm}
\usepackage[latin1]{inputenc}  
\usepackage[T1]{fontenc}       
\usepackage[psamsfonts]{amssymb}

\usepackage{latexsym}
\usepackage[dvips]{epsfig}
\usepackage{dcolumn}
\usepackage{tabularx}
\usepackage{longtable}
\usepackage{graphics}
\usepackage{pstricks}
\usepackage{pst-node}
\usepackage{latexsym}
\usepackage{babel}
\usepackage{euscript}

\renewcommand{\theequation}{\thesection.\arabic{equation}}
\newtheorem{theorem}{Theorem}
\newtheorem{lemma}{Lemma}
\newtheorem{proposition}{Proposition}

\newtheorem{remark}{Remark}
\newtheorem{definition}{Definition}

\newcommand{\eqnsection}{
\renewcommand{\theequation}{\thesection.\arabic{equation}}
    \makeatletter
    \csname  @addtoreset\endcsname{equation}{section}
    \makeatother}
\eqnsection

\def\rrr{{\cal R}}

\def\demi{\frac{1}{2}}

\def\Z{{{\mathbb{Z}}}}
\def\P{{{\mathbb{P}}}}
\def\N{{\mathbb{N}}}
\def\R{{\mathbb{R}}}

\def\fff{{{\mathcal F}}}

\def\ttt{{\mathcal T}}

\def\E{{\mathbb{E}}}
\def\dive{{\hbox{div}}}

\def\hhh{{{\mathcal H}}}

\def\indic{{{\mathbbm 1}}}

\newtheorem{corol}{Corollary}[section]
\newtheorem{rema}{Remark}[section]
\newtheorem{exa}{Example}

\renewcommand{\le}{\leqslant}
\renewcommand{\leq}{\leqslant}
\renewcommand{\ge}{\geqslant}
\renewcommand{\geq}{\geqslant}
\renewcommand{\subset}{\subseteq}

\newcommand{\bal}{\begin{align*}}
\newcommand{\eal}{\end{align*}}
\newcommand{\beq}{\begin{eqnarray*}}
\newcommand{\eeq}{\end{eqnarray*}}
\newcommand{\bte}{\begin{theorem}}
\newcommand{\ete}{\end{theorem}}
\newcommand{\bl}{\begin{lemma}}
\newcommand{\el}{\end{lemma}}
\newcommand{\bd}{\begin{description}}
\newcommand{\ed}{\end{description}}

\newcommand{\bc}{\begin{cases}}
\newcommand{\ec}{\end{cases}}
\newcommand{\bp}{\begin{proof}}
\newcommand{\ep}{\end{proof}}
\newcommand{\bco}{\begin{corol}}
\newcommand{\eco}{\end{corol}}

\newcommand{\ochi}{\overline{\chi}}
\newcommand{\tR}{\tilde{R}}

\newcommand{\1}{\hbox{1 \hskip -7pt I}}

\newcommand{\iy}{\infty}
\newcommand{\tx}{\text}

\newcommand{\cqfd}{\ensuremath{\Box}}

\newcommand{\Cst}{\mathsf{Cst}}

\newcommand{\F}{\mathcal{F}}

\newcommand{\Es}{\mathbb{E}}
\newcommand{\Pb}{\mathbb{P}}

\newcommand{\T}{\mathcal{T}}
\newcommand{\Rc}{\mathcal{R}}

\newcommand{\lan}{\langle}
\newcommand{\ran}{\rangle}

\newcommand{\al}{\alpha}

\newcommand{\be}{\beta}

\newcommand{\g}{\gamma}

\newcommand{\de}{\delta}

\newcommand{\ze}{\zeta}

\newcommand{\La}{\Lambda}

\newcommand{\s}{\sigma}

\newcommand{\Om}{\Omega}

\def\bdes{\begin{description}}
\def\edes{\end{description}}

\def\rrr{{\mathcal R}}

\def\opsi{{\overline\psi}}
\def\ggg{{\mathcal G}}

\begin{document}
\author[M. Disertori]{Margherita DISERTORI}
\address{Institute for Applied Mathematics \&
Hausdorff Center for Mathematics, University of Bonn, 
Endenicher Allee 60, 53115  Bonn, Germany}
\email{margherita.disertori@iam.uni-bonn.de} 
\author[C. Sabot]{Christophe SABOT}
\address{Universit\'e de Lyon, Universit\'e Lyon 1,
Institut Camille Jordan, CNRS UMR 5208, 43, Boulevard du 11 novembre 1918,
69622 Villeurbanne Cedex, France} \email{sabot@math.univ-lyon1.fr}
\author[P. Tarr\`es]{Pierre Tarr\`es}
\address{Ceremade, CNRS UMR 7534 and Universit\'e Paris-Dauphine, Place de Lattre de Tassigny, 75775 Paris Cedex 16, France.} \email{tarres@ceremade.dauphine.fr}
\thanks{This work was partly supported by the LABEX project MILYON, ANR project MEMEMO 2 and  \\
ERC Starting Grant CoMBoS}
\begin{abstract}
We show transience of the edge-reinforced random walk (ERRW) for small reinforcement in dimension $d\ge3$. 
This proves the existence of a phase transition between recurrent and transient behavior, thus solving
an open problem stated by Diaconis in 1986.
The argument adapts the proof of quasi-diffusive behavior of the SuSy hyperbolic model for fixed 
conductances by Disertori, Spencer and Zirnbauer \cite{dsz}, using the representation of ERRW as a mixture of vertex-reinforced jump processes (VRJP) with independent gamma conductances, 
and the interpretation of the limit law of VRJP as a supersymmetric (SuSy) hyperbolic sigma model developed by Sabot and Tarr\`es in \cite{sabot-tarres}. 
\end{abstract}
\maketitle

\section{Introduction}
\subsection{Setting and main result.}
Let $(\Om,\F,\Pb)$ be a probability space. Let  $G=(V,E,\sim)$ be a nonoriented connected locally 
finite graph without loops. Let $(a_{e})_{e\in E}$ and $(w_{e})_{e\in E}$ two 
sequences of positive weights associated to each $e\in E$.

Let $(X_n)_{n\in\N}$ be a random process that takes values in $V$, and let $\F_n=\sigma(X_0,\ldots,X_n)$ 
be the filtration of its past. For any $e\in E$, $n\in\N\cup\{\iy\}$, let 
\begin{equation}
Z_n(e)=a_e+ \sum_{k=1}^n\1_{\{\{X_{k-1},X_k\}=e\}}
\end{equation}
be the number of crosses of $e$ up to time $n$ plus the initial weight $a_e$.
Then $(X_n)_{n\in\N}$ is called Edge Reinforced Random Walk (ERRW) with starting point $i_0\in V$ 
and  weights $(a_e)_{e\in E}$,  
if  $X_0=i_0$ and,  for all $n\in\N$, 
\begin{equation}
\label{def-vsirw}
\Pb(X_{n+1}=j~|~\F_n)=\indic_{\{j\sim X_n\}}\frac{Z_n(\{X_n,j\})}
{\sum_{k\sim X_n} Z_n(\{X_n,k\})}.
\end{equation}

On the other hand, let $(Y_t)_{t\ge0}$ be a continuous-time process on $V$, starting at time $0$ at 
some vertex $i_0\in V$. Then  $(Y_t)_{t\ge0}$ is called a 
Vertex-Reinforced Jump Process (VRJP) with starting point $i_{0}$ and weights $(w_e)_{e\in E}$ if $Y_{0}=i_{0}$
and if $Y_{t}=i$ then, conditionally on $(Y_s, s\le t)$, the process jumps to a neighbour $j$ of $i$ at 
rate $w_{\{i,j\}}L_j(t)$, where
$$L_j(t):=1+\int_0^t \1_{\{Y_s=j\}}\,ds.$$

The Edge Reinforced Random Walk was introduced in 1986 by Diaconis \cite{coppersmith} 
and is known to be a mixture of reversible
Markov chains with explicitly known mixing measure \cite{coppersmith,keane-rolles1}. 

On acyclic graphs the model displays a recurrence/transience phase transition, as was first observed by Pemantle \cite{collevecchio1,pemantle4}. On finite strips $\Z\times G$, $G$ finite, Merkl and Rolles \cite{merkl-rolles4} showed that the model is recurrent for any choice of the parameters $a>0$. Recurrence is conjectured to hold also on $\Z^2$ for any reinforcement. A first result in this sense was obtained by
 Merkl and Rolles \cite{merkl-rolles2} in 2009, who showed recurrence on a diluted version of $\Z^2$ (the amount of dilution needed increases with the parameter $a>0$) using a Mermin-Wagner type argument.
More recently almost-sure positive  recurrence for large reinforcement was proved by 
Sabot and Tarr\`es in \cite{sabot-tarres} {\it on any graph of bounded degree }
(i.e. if  $a_e<\tilde{a}_c$ for all $e\in E$,
 for some  $\tilde{a }_c$ depending on that degree), using  localisation results of Disertori and
Spencer  \cite{ds}; an alternative proof was { given}
 by Angel, Crawford and Kozma \cite{ack}.

The Vertex-Reinforced Jump
 Process was proposed by Werner in 2000, and initially studied by Davis and Volkov \cite{dv1,dv2}; 
for more details 
on {this model} and related questions, see \cite{sabot-tarres} for instance. 
{As for ERRW,  almost-sure positive recurrence for large reinforcement 
(i.e. if   $w_e<\tilde{w}_c$ for all $e\in E$,
 for some  $\tilde{w}_{c}>0$) on any graph of bounded degree, was proved in   \cite{sabot-tarres} and  \cite{ack}.
Moreover,  transience  for small reinforcement (i.e. if  $w_e>w_c$ for some $w_c>0$) on $\Z^d$,  $d\geq 3$, was proved by 
Sabot and Tarr\`es in \cite{sabot-tarres} using  delocalisation results of Disertori,
Spencer and Zirnbauer \cite{dsz}.}

The aim of this paper is to prove transience of the edge-reinforced random walk (ERRW) for large $a_e>0$, i.e. 
small reinforcement, in any dimension larger than two.
This proves the existence of a phase transition between recurrent and transient behavior, thus solving
an open problem stated by Diaconis in 1986.

\begin{theorem}
\label{thm:transience}
On $\Z^d$, $d\ge3$, there exists $a_c(d)>0$ such that, if $a_e>a_c(d)$ for all $e\in E$, 
then the ERRW with weights $(a_e)_{e\in E}$ is transient a.s.
\end{theorem}

The proof of Theorem \ref{thm:transience} follows from estimates on the fluctuation of a field $(U_i)$ 
associated to the limiting behaviour of the reinforced-random walk.  { The key tools are
an inductive argument on scales and a family of Ward identites generated by some internal symmetries 
of the probability measure.}
 Let us first recall two earlier 
results from Sabot and Tarr\`es \cite{sabot-tarres}. 

\begin{theorem}[Sabot and Tarr\`es \cite{sabot-tarres}]
\label{annealed}
On any locally finite graph $G$, the ERRW $(X_n)_{n\ge 0}$ is equal in  law to the discrete time process 
associated with a 
VRJP in random independent conductances $W_e\sim \it{Gamma}(a_e,1)$.
\end{theorem}

The next result concerns VRJP $(Y_t)_{t\ge0}$ on a finite graph $G$, with $|V|=N$, given fixed weights
  $(w_e)_{e\in E}$; let $\P^{VRJP}_{i_0}$ be its law, starting from $i_0\in V$. Note that the two results 
below are not stated exactly in the same terms as in
 proposition 1 and theorem 2 of \cite{sabot-tarres} but can be obtained by the following simple transformation :
  $U_i$ (resp. $\log L_i$) here corresponds to $U_i-U_{i_0} $ (resp.  $T_i$) of \cite{sabot-tarres}.  
\begin{proposition}[Sabot and Tarr\`es \cite{sabot-tarres}]
\label{pconv}
Suppose that $G$ is finite and set $N=\vert V \vert$.  For all $i\in V$, the following limits exist 
 $\P^{VRJP}_{i_0}$ a.s.
$$
U_i =\lim_{t\to \infty} \left(\log L_i(t) - \log L_{i_0}(t)\right).
$$
\end{proposition}
\begin{theorem}[Sabot and Tarr\`es \cite{sabot-tarres}]
\label{meas}
{\bf(i)} Under  $\P^{VRJP}_{i_0}$, $(U_i)_{i\in V}$ has the following density with respect to the Lebesgue measure on  
$$\hhh_0=\{(u_i)\in\R^V: \; u_{i_0}=0\}$$
\begin{eqnarray}\label{density}
 d\rho_{w,V} (u)= \frac{1}{(2\pi)^{(N-1)/2}} e^{-\sum_{j\in V} u_j} 
e^{-H(w,u)} \sqrt{D[m(w,u)]}\prod_{j\in V\setminus\{i_0\}}du_j
\end{eqnarray}
where
$$
H( w,u)= \sum_{\{i,j\}\in E}  w_{i,j} (\cosh(u_i-u_j)-1)
$$
and $D[m(w,u)]$ is any diagonal minor of the $N\times N$ matrix $m(w,u)$ with coefficients
$$
m_{i,j}(w,u)=\left\{ \begin{array}{ll} {-  w_{i,j}} e^{u_i+u_j}  & \hbox{ if $i\neq j$}
\\
\sum_{k\in V}  w_{i,k} e^{u_i+u_k} &\hbox{ if $i=j$}
\end{array}\right.
$$

\noindent{\bf(ii)}  Let $C$ be the following positive continuous { increasing} functional of $X$:
$$C(s)=\sum_{i\in V} L_i^2(s)-1,$$
and let $$Z_t=Y_{ C^{-1}(t)}.$$ 
 Then, conditionally on $(U_i)_{i\in V}$, $Z_t$  is a Markov jump process starting from  $i_0$, with jump rate 
from  $i$ to $j$
$$
\demi w_{i,j}e^{U_j-U_i}.
$$
In particular, the discrete time process associated with $(Y_s)_{s\ge0}$ is a mixture of reversible Markov chains 
with conductances $ w_{i,j} e^{U_i+U_j}$.
\end{theorem}
\begin{remark}
The diagonal minors of the matrix $m(w,u)$ are all equal since the sums on any line or column
of the coefficients of the matrix are null. By the matrix-tree theorem, if we let $\T$ be the set of 
spanning trees of $(V,E,\sim)$, then $D[m(w,u)]=\sum_{T\in\T}\prod_{\{i,j\}\in\T}w_{\{i,j\}}e^{u_i+u_j}$.
\end{remark}

\paragraph{\bf Notation and conventions}
In the sequel we fix $d\ge3$.

A sequence $\sigma=(x_0, \ldots, x_n)$ is a path from
$x$ to $y$ in $\Z^d$ if $x_0=x$, $x_n=y$ and $x_{i+1}\sim x_i$
for all $i=1, \ldots, n$. 

Let
$$V_n=\{i\in \Z^d, \; | i|_\iy\le n\}$$ 
be the ball centered at $0$ with radius $n$, and let 
$$\partial  V_n=\{i\in \Z^d, \; | i|_\iy = n\}$$
be its boundary. We denote by $E$ the set of edges in $ \Z^d$ and by $E_n$ 
the set of edges contained in the hypercube $ V_n$.  We denote by $\tilde E_n$ 
the associated set of directed edges.

In the rest of the paper we will denote by $|.|$ the $L^2$ norm on $\R^d$.

Let  $(a_{ij})_{i,j\in\Z^d, i\sim j}$ be the  family of initial positive weights for the ERRW, and let
\begin{equation}\label{adef}
a:=\inf_{e\in E}a_e.
\end{equation}
In the proofs, we will denote by  $\Cst(a_1,a_2,\ldots, a_p)$ a positive constant depending
only on  $a_1$, $a_2$, $\ldots$ $a_p$, and by $\Cst$  a  universal
positive constant.

In the sequel we fix $i_{0}=0$. Let $\Pb_0$ (resp. $\Pb_0^{V_n}$) be the law of the ERRW on $\Z^d $ (resp. on $V_{n}$).
Theorem \ref{meas} ensures that on every finite volume $V_{n}$  the ERRW is a mixture of 
reversible Markov chains with random conductances $(W^{U}_{e})_{e\in E_{n}}$ where
$W_{ij}^{U}= W_{ij}e^{U_{i}+U_{j}}$ and the law for the random variables $(W,U)$ has 
joint distribution on $\R_+^{E_n}\times \hhh_0$ given by
\[
d\rho_{V_{n}} (w,u) =   d\rho_{w,V_{n}} (u)  
\prod_{e\in E_{n}} \frac{e^{-w_{e}} w_{e}^{a_{e}-1}}{\Gamma (a_{e})}dw_e,
\]
where $ d\rho_{w,V_{n}} (u) $ was defined in \eqref{density}. 
Let $ \E_0^{V_{n}}$ be the average with respect to the joint law for $(W,U)$  (mixing measure).
Note that we cannot define this average on an infinite volume, since we do not know if the
limiting measure exists. Denote by  $\lan\cdot\ran_{V_n}$ the corresponding marginal in $U$.

\paragraph{\bf Warning} We use a capital letter $W, U\dotsb $ to denote a random variable
and a smallcase letter $w,u,\dotsc $ to denote a particular realization of the variable.  The same
is true for any function of such variables. In some cases though we do not state the argument explicitely, 
to avoid heavy notations. It should become clear from the context when the corresponding argument 
has to be regarded as a random variable.

The proof of Theorem \ref{thm:transience} will follow from the following result. 
\begin{theorem}
\label{thm:fluctuations}
Fix $d\ge3$. For all $m>0$, there exists  $a_c(m,d)>0$ such that, if $a\ge a_c(m,d)$ then, 
for all $n\in\N$, $x$, $y$ $\in V_n$, 
$$\lan\cosh^m(U_x-U_y)\ran_{V_n}\le2.$$
\end{theorem}

The proof of Theorem \ref{thm:fluctuations} is the purpose of the rest of this paper, and adapts the argument 
of Disertori, Spencer and Zirnbauer \cite{dsz}, which implied in particular transience of VRJP with large 
conductances $w_e$ \cite{sabot-tarres}.  Let us first show how Theorem \ref{thm:fluctuations} implies 
Theorem \ref{thm:transience}.

\paragraph{\bf Proof  of Theorem \ref{thm:transience}.}
Given  $(w, u)\in\R_+^E\times\hhh_0$, denote by $P_0^{w^{u}}$ the law of the Markov chain in conductances 
$w_{i,j}^{u}= w_{ij} e^{u_i+u_j}$ starting from $0$.
Let $H_{\partial V_n}$ be the first hitting time of the boundary $\partial  V_n$, and let 
$\tilde{H}_0$ be the first return time to the point $0$.  

We seek to estimate  $\Pb_0 ( H_{\partial V_n} <\tilde{H}_0)$ the probability that the ERRW
hits the boundary of $V_{n}$ before coming back to $0$. Since this event depends only
on the history of the walk inside the hypercube  $V_{n}$, we have
\begin{equation}\label{pbes}
\Pb_0 ( H_{\partial V_n} <\tilde{H}_0)= \Pb_0^{V_n} ( H_{\partial V_n} <\tilde{H}_0)=
\Es_0^{V_{n}}\left[ P^{W^{U}}_0 (H_{\partial  V_n}<\tilde{H}_0)\right]
\end{equation}
where in the last equality we used the last part of Theorem \ref{meas}. Now 
$P^{w^{u}}_0 (H_{\partial  V_n}<\tilde{H}_0)$ is the probability for a Markov chain inside $V_{n}$
with conductances $w^{u}$ to hit the boundary before coming back to $0$.
This probability is related, by identity \eqref{rpe} below, to the  effective resistance between $0$  
and $\partial V_n$ (see for instance in Chapter 2 of \cite{lyons-peres})
\[
R(0,\partial V_n,w^{u})= \inf_{\theta:\tilde E_{n} \to \R} \demi\sum_{e\in \tilde E_{n}} \frac{ \theta (e)^{2}  }{ w^{u}_{e}}
\]
where the infimum is taken on unit flows $\theta$ from 0 to $\partial  V_n$ : $\theta $ is defined on the set of directed edges $\tilde E_n$ 
and must satisfy $\theta ((i,j))=-\theta ((j,i))$ and $\sum_{j\sim v}\theta ((v,j))=\delta_{v,0}$
for all $v\in V_{n}$. Denote by  $R(0,\partial V_n)$ 
the effective resistance between $0$ and $\partial V_n$ for conductances $1$.

Classically we have
\begin{equation}
\label{rpe}
 w_{0}^{u} R(0, \partial  V_n,w^{u}) =\frac{1}{ P^{w^{u}}_0 (H_{\partial  V_n}<\tilde H_0) }
\end{equation}
with $w_{0}^{u}=\sum_{j\sim 0} w_{0,j}^{u}$.
  Using \eqref{pbes} and   Jensen's inequality, we deduce
\begin{align}
&\frac{1}{\Pb_0 ( H_{\partial V_n} <\tilde H_0)}= 
\frac{1}{ \E_0^{V_{n} } \left[ P^{W^{U}}_0 (H_{\partial  V_n}<\tilde H_0)\right] }\cr
&\qquad \qquad  \leq 
\Es_0^{V_{n} }\left[ \frac{1}{P^{W^{U}}_0 (H_{\partial  V_n}<\tilde H_0) }\right]
= 
 \Es_0^{ V_n}\left[ W_0^{U} R(0, \partial  V_n,W^{U})\right].
 \label{inbd}
\end{align}
We will show below that, if $\min_{e\in E} a_{e}=a> \max\{a_c(3,d),3 \}$, then
\begin{eqnarray}\label{Resistance}
\Es_0^{ V_n}\left[ W_0^{U} R(0,\partial V_n,W^{U})\right]\le\Cst(a,d) R(0,\partial V_n)
\end{eqnarray}
This will enable us to conclude: since $\limsup R(0,\partial  V_n) <\infty$ for all $d\ge3$, 
\eqref{inbd} 
and \eqref{Resistance} imply that $ \Pb_0( \tilde H_0 =\infty)>0$. 
\cqfd
\paragraph{\bf Proof of \eqref{Resistance}}
Let $\theta$ be the unit flow from 0 to $\partial  V_n$ which minimizes the $L^2$ norm. Then
$$
R(0,\partial  V_n, w_{u}) \le \demi\sum_{(i,j)\in \tilde E_n}\frac{1}{w_{i,j}^{u} }  \theta^2(i,j),
$$
and 
$$
R(0,\partial  V_n) = \demi\sum_{(i,j)\in \tilde E_n} \theta^2(i,j).
$$
Now, if $a>\max\{\tilde{a}_{c} (3,d),3 \}$ then, using Theorem \ref{thm:fluctuations},
\begin{align*}
 \Es_0^{ V_n}\left[\frac{W_0^{U} }{W_{i,j}^{U}}\right] &=
 \sum_{l\sim0}  \Es_0^{ V_n}\left[ \frac{W_{0,l} }{W_{i,j}} e^{U_{0}-U_{i}} e^{U_{l}-U_{j}} \right]
\\
&\le \sum_{l\sim0, \{i,j\}\ne\{0,l\}}
\left[\left (  \Es_0^{ V_n}\left[ \frac{W_{0,l}^3 }{W_{i,j}^3} \right]
\  \Es_0^{ V_n}\left[  e^{3 (U_{0}-U_{i})} \right] 
\Es_0^{ V_n}\left[  e^{3 (U_{l}-U_{j})} \right]\right)^{1/3} + \1_{{ \{i,j\}=\{0,l\}}}  \right]
\cr
&
=  \sum_{l\sim0,\{i,j\}\ne\{0,l\}} \left[\left ( \Es_0^{ V_n}\left[ W_{0,l}^3  \right] 
\Es_0^{ V_n}\left[ W_{i,j}^{-3} \right]
\lan e^{3(U_0-U_i)}\ran_{ V_n}  \lan e^{3(U_l-U_j)} \ran_{ V_n}\right)^{1/3} +\1_{{ \{i,j\}=\{0,l\}}}  
\right]
 \\
 &\le\Cst(a,d),
 \end{align*}
 where we also used the fact that $W_{i,j}$ are independent Gamma distributed random variables. 
\cqfd

\subsection{Organization of the paper}
Section \ref{sketch} introduces to the main lines of the proof, the geometric objects that are needed 
(diamonds and deformed diamonds), and states some Ward identities. 
Section \ref{good}  { introduces some definitions and estimates (in particular the notion of ``good point'') 
that will be used in the proof.}
Section \ref{induction} contains the main inductive argument. 
The Ward identities { we need} (Lemmas \ref{ward} and \ref{wardp}) are shown in Section   
\ref{ward-resist}, and { finally}  the estimate on the effective conductance given in 
Proposition \ref{resistance-bound} is proved in Section \ref{sec:res}.

\vskip1.truecm

\noindent {\bf \textit{ACKNOWLEDGMENTS.}} 
It is our pleasure to thank T. Spencer for very important suggestions on this work
and many useful discussions.
Margherita Disertori would like to thank Martin Zirnbauer for introducing her to the $H^{2|2}$ model
that plays a crucial role in this paper. She also thanks the Institute for Advanced Study for 
their hospitality while some of this work was in progress. Pierre Tarr\`es and 
Christophe Sabot are particularly grateful to Krzysztof Gawedzki for a helpful discussion on the 
hyperbolic sigma model, and for pointing out reference \cite{dsz}.  

\section{Introduction to the Proof}
\label{sketch}
\subsection{Marginals of $U$ and a first Ward identity}
Let us now fix $n\in\N$, and let $\La=V_n$, $E=E_n$, for simplicity. 

A key step in our proof is to study the law of $U$ after integration over the conductances $W_e$, in other words to focus on the marginal $\lan.\ran_\La$. Note that it is one of several possible approaches to the question of recurrence/transience for ERRW. Indeed, one could instead focus on the analysis of the law of the random conductances $W_{ij}e^{U_i+U_j}$ given by Coppersmith-Diaconis formula, or directly study ERRW from its tree of discovery using the existence of a limiting environment, as done in \cite{ack}, or possibly conclude from a.s. results conditionally on the conductances $W_e$, $e\in E$. So far these other approaches have not proved to 
be efficient to understand the delocalized phase on $\Z^d$. A key ingredient in our approach is that it enables to use the strength of Ward estimates as \cite{dsz}.

In fact we add a Gaussian Free Field variable $S$ with conductances 
$W_{ij}^{U}$ before integrating over $W$, as in the first step of \cite{dsz}, since the corresponding joint law again has more transparent symmetries
and is better suited for the subsequent analysis. 
More precisely, let 
\begin{equation}
\label{def:bij}
B_{xy}:=\cosh(U_x-U_y)+\demi e^{U_x+U_y}(S_x-S_y)^2
\end{equation}
for all $x$, $y$ $\in \Lambda$ (not necessarily neighbours).

Then $(W,U,S)$ has the following probability density 
on $\R_+^{E}\times\hhh_0\times\hhh_0$  
\begin{equation}\label{usold}
\tfrac{1}{(2\pi)^{(N-1)}} e^{-\sum_{j\in\La}u_{j}}  e^{-\sum_{e\in E}w_{e}(B_{e}-1)}
D[m(w,u)]
\ \prod_{k\ne 0} du_k ds_k\prod_{e\in E} \frac{e^{-w_{e}} w_{e}^{a_{e}-1}}{\Gamma (a_{e})}dw_e
\end{equation}
The marginal of this law in $(U,S)$ after integration over $W$, which we still call $\lan.\ran_\La$ or $\lan.\ran$ by a slight abuse of notation, is given  in the following Proposition \ref{newm}. 
\begin{proposition}\label{newm}
The joint variables $(U,S)$ have density on $\hhh_0\times\hhh_0$
\begin{equation}\label{usmes}
\mu_{a,\Lambda} (u,s)=\frac{1}{(2\pi)^{(N-1)}}\left[\prod_{e} \frac{1}{B_e^{a_e}} \right]
 e^{-\sum_{j\in\La}u_{j}} D[M(u,s)],
\end{equation}
where $D[M(u,s)]$ is any diagonal minor of the $N\times N$  matrix  $M(u,s)$ defined by
 $$
 M_{i,j}= \left\{\begin{array}{ll} {-c_{i,j}}, &\hbox{ $i\neq j$}
 \\
\sum_k{c_{i,k}}, &\hbox{ $i=j$}
\end{array}\right.
$$
and
\[
c_{i,j}:=\frac{a_{i,j}e^{u_i+u_j}}{ B_{i,j}}, \quad\forall , i\sim j
\]
Note that we slightly abuse notation here, 
in that $B_{i,j}$ is considered alternatively as function of $(u,s)$ and $(U,S)$. 
\end{proposition}
\paragraph{\bf Proof}
Integrating the density in  \eqref{usold} with respect to the  random independent conductances 
$W_e\sim \it{Gamma}(a_e,1)$
gives
\begin{align*}
&\tfrac{ e^{-\sum_{j}u_{j}} }{(2\pi)^{(N-1)}} \int  
\prod_{e} \frac{e^{-w_{e}} w_{e}^{a_{e}-1}dw_{e}}{\Gamma (a_{e})}
 e^{-\sum_{e\in E}w_{e}(B_{e}-1)}
D[m(w,u)]\cr
&=
\tfrac{ e^{-\sum_{j}u_{j}}}{(2\pi)^{(N-1)}}  
\sum_{T} \prod_{i\sim j\in T} e^{u_{i}+u_{j}} 
\left[\prod_{e\not \in T} I_{e} (a_{e})\right] \left[\prod_{e\in T} I_{e} (a_{e}+1)\right]\cr
&=\left[\prod_e \tfrac{1}{B_{e}^{a_{e}}}\right] \tfrac{ e^{-\sum_{j}u_{j}}}{(2\pi)^{(N-1)}}  
\sum_{T} \prod_{i\sim j\in T} 
\tfrac{a_{ij}e^{u_{i}+u_{j}}  }{B_{ij}}  = \mu_{a,\Lambda } (u,s)
\end{align*}
where in the second line we expand the determinant as a sum  over spanning trees and, for all $x>0$, 
$$I_{e}(x)=\frac{1}{\Gamma (a_{e})}\int_0^\iy  e^{-w B_{e}} w^{x-1}\,dw=\frac{\Gamma (x)}{\Gamma (a_e)B_{e}^{x}},$$
with in particular $I_{e} (a_{e})= B_{e}^{-a_{e}}$ and $I_{e} (a_{e}+1)= a_{e}/B_{e}^{a_{e}+1}.$
\qed

In the proof of delocalization in  \cite{dsz}, the first step  was a set of relations (Ward identities)
generated by internal symmetries.  Their analogue is  stated in Lemma \ref{ward}  (proved 
in Section \ref{ward-resist}), and involves a term of effective resistance 
$D_{x,y}$ depending on the variable $U$, defined as follows.
\begin{definition}
\label{def:er}
Let $D_{x,y}$ be the effective resistance between $x$ and $y$ for the conductances
$$
c_{i,j}^{x,y}:=c_{i,j}B_{x,y}e^{-u_x-u_y}= \frac{ a_{i,j}e^{u_i+u_j-u_x-u_y}B_{x,y}}{B_{i,j}}.$$ 
\end{definition}
 
\begin{lemma}
\label{ward}
For all $m\leq a/4=\min_{ij}a_{ij}/4$, for all $x$, $y$ $\in \La$ (not necessarily neighbours),  
$$
\lan B_{x,y}^m(1-mD_{x,y})\ran=1.
$$
\end{lemma}
One of the consequences of Lemma \ref{ward} is that, if the effective resistance $D_{xy}$ is small, 
then we can deduce a good bound on $B_{x,y}^m$. However, there is a positive probability that this 
resistance $D_{xy}$ might be large, so that it would be more useful to show an inequality restricted to the event that $D_{xy}$ is small. 

Although it is not possible to derive such an identity directly, we can define events $\ochi_{xy}$ on which such identities hold, and so that $D_{xy}$ is small on  $\ochi_{xy}$. These events will depend on geometric objects defined in the next Section \ref{geometry}, called diamonds and deformed diamonds.

\subsection{Diamonds and deformed diamonds}
\label{geometry}
\begin{definition}
Let $l\neq 0$ be a vector in $\R^d$, and let $x\in \Z^d$. We denote by $C_x^l$ the cone with base $x$, direction $l$ and angle $\pi/4$,
\beq
C_x^l &=& \{ z\in \R^d, \;\;\; \angle (xz, l) \leq \frac{\pi}{4} \}\\
&=&
\{ z\in \R^d, \;\;\; (z-x)\cdot l \geq \frac{\sqrt 2}{2} | z-x|| l |\}
\eeq
where $\angle (xz, l) $ is the angle between the vector $xz$ and the vector $l$. Recall that $|.|$ is the $L^2$ norm on $\R^d$.

If $x$ and $y$ are in $\Z^d$, we call Diamond the set
$$
\left(C^{y-x}_x\cap C_y^{x-y}\right)\cap \Z^d.
$$
to which we add a few points close to $x$ and close to $y$ so that the set is connected in $\Z^d$. We denote this set by $R_{x,y}$.
\end{definition}
\begin{remark}\label{points}
By the expression "a few points" above we mean that we add some extra points to 
$\left(C^{y-x}_x\cap C_y^{x-y}\right)\cap \Z^d$ at a bounded distance from $x$, $y$ so that the resulting set becomes connected in the lattice.
The distance at which the points can be added is bounded by a constant depending only on the dimension.  Of course, all the estimates below will be independent
of the choice of these points. We will repeat this operation several times
hereafter, without extra explanation.
\end{remark}
In the course of the inductive argument, some deformed diamonds appear, they are formed of the intersection of two cones with smaller angles than
for diamonds. For $l\in \R^d$, $l\neq 0$, and $x\in \R^d$ we set
$$
\tilde C_x^l = \{ z\in \R^d, \;\;\; \angle (xz, l) \leq \frac{\pi}{16} \}.
$$
\begin{definition} \label{deformed}
A deformed diamond is a set of the following form
$$
\left( \tilde C_x^l \cap \tilde C_y^{x-y}\right)\cap \Z^d,
$$
(plus a few points close to $x$ and to $y$ so that the set is connected in $\Z^d$, see Remark \ref{points}) where
$x\in \Z^d$, $l\in \R^d$,  $l\neq 0$  and $y\in \Z^d$ is a point such that
$$
y\in \tilde C^l_x.
$$
We also denote a deformed diamond by $R_{x,y}$ (and it will always be clear in the text wether $R_{x,y}$ is a diamond or deformed diamond).
\end{definition}
\begin{definition}
\label{deformedf}It will be useful to write each (exact or deformed) diamond as a non disjoint union of two sets 
$$
R_{x,y}^x=\{z\in R_{x,y}, \; | z-x | \leq f_x | y-x|\}, \;\;
R_{x,y}^y=\{z\in R_{x,y}, \; | z-y | \leq f_y | y-x|\},
$$
where the pair $(f_x, f_y)$ is such that $ \frac{1}{5} \le f_x \leq 1$, $\frac{1}{5}\leq f_y \leq 1$ and 
$f_x + f_y\geq 1+\frac{1}{5}$.

This condition on $(f_x,f_y)$, and the fact that the maximal angle in deformed diamonds is $\pi/16$, indeed ensures that 
$R_{x,y} =R_{x,y}^x\cup R_{x,y}^y$. Note that the choice of $(f_x,f_y)$ is not unique.
\end{definition}

\subsection{Estimates on effective conductances}

\begin{definition}
Given a diamond or deformed diamond $R_{xy}\subset\La$, let $D_{xy}^N$ be the effective resistance of 
the electrical network with the same conductances $c_{i,j}^{x,y}$ in $R_{xy}$ as in Definition \ref{def:er}, 
and Neumann boundary conditions on $\partial R_{xy}$.
\end{definition}

\begin{definition}\label{defchiij}
Fix $b>1$ and $\al\ge0$. Given $i$, $j$ $\in\La$, let 
$$\chi_{ij}=\indic_{\{B_{ij}\le b|i-j|^{\al}\}}.$$

Given a deformed diamond $R_{xy}\subset\La$ and the two corresponding regions $R_{xy}^x$ and $R_{xy}^y$, let 
$$\ochi_{xy}=\prod_{j\in R_{xy}^x}\chi_{xj}\prod_{j\in R_{xy}^y}\chi_{yj}.$$

Note that $\ochi_{xy}$ implicitly depends on $f_x$ and $f_y$, through $R_{xy}^x$ and $R_{xy}^y$.
\end{definition}
Here is the main proposition of the section. It ensures that under the condition $\ochi_{xy}$, there is a uniform bound on the effective resistance $D^N_{x,y}$.
\begin{proposition}\label{resistance-bound}
Fix $\al\in [0,1/8]$ and $b>1$. There exists a constant $C=\Cst(d,b)$ such that for any deformed diamond 
$R_{x,y}\subset \La$, if $\ochi_{xy}$ is satisfied, then
$$
D^N_{x,y}\le C/a,
$$ 
where $a=\inf (a_{i,j})$ {was introduced in \eqref{adef}.}
\end{proposition}
\begin{remark} Note that the constant $C$ is independent of the precise shape of the deformed diamond. 
\end{remark}

Proposition \ref{resistance-bound} is proved in Section \ref{sec:res}. It partly relies on the following two Lemmas \ref{Bxyz} and \ref{estcond}, which will also be useful in other parts of the proof.
\begin{lemma}[lemma 2 of \cite{dsz}]
\label{Bxyz}
For all $x$, $y$, $z$ $\in \La$, 
$$B_{xz}\le2 B_{xy}B_{yz}.$$
\end{lemma}
\begin{proof}
Elementary computation.
\end{proof}
\begin{lemma}
\label{estcond}
For all $i$, $j$ $\in\Lambda $, $i\sim j$, $x$, $y$ $\in\Lambda $, we have
\begin{equation}
\label{condup}
(c_{ij}^{xy}/a)^{-1}\le16(B_{iz})^2(B_{jz})^2,
\end{equation}
for both $z=x$ and $z=y$.

In particular, let $c=b^{-4}/64$, $\be=4\al$. Assume that $\ochi_{xy}$ holds, then the electrical 
network with conductances
\begin{equation}\label{gammadef}
(\gamma_{i,j})_{i,j\in\La,\,i\sim j}=(c_{ij}^{xy}/a)_{i,j\in\La,\,i\sim j},
\end{equation}
satisfies
$$
\g_{i,j}\ge c|i-z|^{-\be}
$$
for all $i$, $j$ $\in R_{xy}^z$, $i\sim j$, for $z=x,y$,.
\end{lemma}


{\it Proof of Lemma \ref{estcond}. }
First note that 
$$e^{\pm(u_x-u_y)}B_{xy}\ge e^{\pm(u_x-u_y)}\cosh(u_x-u_y)\ge1/2,$$
so that 
\begin{equation}
\label{in1}
e^{u_i+u_j-u_x-u_y}B_{x,y}\ge\frac{1}{2}\max[e^{u_i+u_j-2u_x},e^{u_i+u_j-2u_y}].
\end{equation}
Now
\begin{equation}
\label{in2}
e^{u_i+u_j-2u_x}=e^{u_i-u_x}e^{u_j-u_x}\ge\frac{1}{4}(\cosh(u_i-u_x)\cosh(u_j-u_x))^{-1}\ge\frac{1}{4}(B_{ix}B_{jx})^{-1}.
\end{equation}

On the other hand, note that
\begin{equation}
\label{in3}
B_{ij}\le2B_{ix}B_{jx}.
\end{equation}
Inequalities \eqref{in1}--\eqref{in3} together yield the first part. The second part is a consequence of $B_{ix}\le b|i-x|^\al$ and $B_{jx}\le b|j-x|^\al\le2b|i-x|^\al$.
\cqfd

\subsection{Protected Ward estimates}

Next, we obtain a ``protected'' Ward estimate, as follows. 
\begin{lemma}
\label{wardp}Let $\al\ge0$ and $b>1$.
For all $i=1,\ldots, n$, let $m_i\leq a/4$ {(with $a$ given in \eqref{adef})} 
and let  $R_{x_iy_i}$ be regions whose interiors are disjoint. Then
\begin{equation}
\label{protect1}
\lan \prod_{i=1}^n B_{x_iy_i}^{m_i} (1-m_iD_{x_iy_i}^N)\ran\le 1, 
\end{equation}
and
\begin{equation}
\label{protect2}
\lan \prod_{i=1}^n B_{x_iy_i}^{m_i}\ochi_{x_iy_i} (1-m_iD_{x_iy_i}^N)\ran\le 1.
\end{equation}

Inequality \eqref{protect2} implies, by  Proposition \ref{estcond}, that if, additionally, the regions $R_{x_iy_i}$, $i=1,\ldots, n$ are deformed diamonds and $m_i<a/C$ for all $i$, then 
\begin{equation}
\label{protect3}
\lan\prod_{j=1}^n B_{x_jy_j}^{m_j}\ochi_{x_jy_j}\ran\le\prod_{j=1}^n(1-m_jC/a)^{{ -1}},
\end{equation}
where $C$ is the constant considered in Proposition \ref{resistance-bound}.
\end{lemma}
Lemma \ref{ward} is proved in Section \ref{ward-resist}. The rest of the proof is similar to the argument in \cite{dsz}, which consists in deducing upper bounds of $\lan\prod_{j=1}^n B_{x_jy_j}^m\ran$ for some fixed $m$ from these ``protected'' estimates, through an induction on the maximal length $|x_i-y_i|$ on the protected estimates, $i=1$ $\ldots$ $n$, using Chebyshev inequalities in order to deal with the ``unprotected'' parts of the estimates. We summarise the argument in Section \ref{induction}.

\section{Estimates on good points}
\label{good}
\begin{definition}
\label{def:good}
A point $x\in \Lambda$ is called $n$-good if
$$
B_{x,y}\le b \vert x-y\vert^\alpha,
$$
for all $y\in \Lambda $ with distance $1\le \vert x -y \vert \le 4^n$ from $x$.

Given $z\in\La$, let $R_n(z)$ (denoted by $R_n$ when there is no ambiguity) the  hypercube with side $4^n$ and barycenter $z$. We denote by $\chi^c_{R_n}$ the indicator of the event that there is no $n$-good point in $R_n$.
\end{definition}

In this section we present an estimate on the event $\chi_{R_n}^c$, more precisely we bound from above the indicator function by a sum of terms involving 
$\frac{B_{i,j}}{b\vert j-i\vert^\alpha}$ for points $i,j$ at distance at most $4^n$. 
This will be used in the main inductive argument in Section \ref{induction}. It follows from lemma
9 of \cite{dsz}, but, for convenience of the reader, we give the proof which is only a few lines long.

If $R_n$ is a hypercube of side $4^n$ then it is the disjoint union of $4^d$ sub-hypercubes of side $4^{n-1}$. We can select the $2^d$ corner subcubes
that we denote $(R_n^{i}, \; i=1, \ldots , 2^d)$ so that $d(R_n^i, R_n^j)\ge2 \times 4^{n-1}$ for $i\neq j$. Repeating this procedure hierarchically we can construct a family of
hypercubes $(R_n^v)$ with side $4^{n-k}$ for $v$ running on the set $\{1, \ldots, 2^d\}^k$. We consider the natural structure of rooted $(2^d+1)$-regular tree on the set
$$
\rrr=\{root\}\cup\left(\cup_{k=1}^n \{1, \ldots, 2^d\}^k\right),
$$
where "root" is the root of the tree, corresponding the cube $R_n(=R_n^{root})$. We denote by $d_v$ the depth of a point $v\in \rrr$
with $d_{root}=0$ and $d_v=k$ if $v\in \{1, \ldots, 2^d\}^k$. 
We denote by $\ttt_n$ the set of connected subtrees $T$ of $\rrr$ containing the root and which have the property that any vertex $x\in T$ have
either 0 or $2d$ descendants in $T$. We denote by $L_T$ the set of leaves of such a tree $T\in \ttt_n$.  (Remark that the set $\{L_T, T\in \ttt_n\}$ is also the set of
maximal totally unordered subsets of $\rrr$ for the natural "genealogical order" on $\rrr$.)

For an element $v\in \rrr$ with  $k=d_v$ we set
$$
S^c_{R_n^v}=\mathop{ \sum_{x\in R_n^v, \; y\in \Lambda}}_{ 4^{n-k-1}<\vert x-y\vert <4^{n-k}} \chi_{xy}^c\le 
\mathop{\sum_{x\in R_n^v, \; y\in \Lambda}}_{ 4^{n-k-1}<\vert x-y\vert <4^{n-k} } \frac{B_{x,y}^m}{b^m\vert x-y\vert^{\alpha m}},
$$
{ where $\chi_{xy}^c=1-\chi_{xy}$.}
\begin{lemma}(lemma 9 of \cite{dsz}) \label{Rn}
With the notations above we have
$$
\chi_{R_n}^c \le \sum_{T\in \ttt_n} \prod_{v\in L_T} S^c_{R_n^v}.
$$
\end{lemma}
\begin{proof}
If there is no $n$-good point in $R_n$ then either there is no $(n-1)$-good point in any of the subcubes $\{R_n^i, i=1, \ldots, 2^d\}$ or else
there exists at least one pair $(x,y)\in R_n\times \Lambda$ with $4^{n-1}<\vert x-y\vert < 4^n$ and $B_{x,y}>b \vert x-y\vert^{\alpha}$. This gives the first
level inequality
$$
\chi_{R_n}^c \le S_{R_n}^c+ \prod_{i=1}^{2^d} \chi^c_{R^i_{n-1}}.
$$
Then the proof follows by induction on the integer $n$. For $n=0$, $R_0(z)$ is the singleton $z$ and $\chi_{R_0}^c= S^c_{R_0}=0$ which initializes the induction.
If Lemma \ref{Rn} is valid at level $n-1$ then, obviously, the previous inequality implies it at level $n$. 
\end{proof}  

\section{Inductive argument}
\label{induction}
Fix $b>1$ and $\al\in (0,1/8)$, 
which will be chosen later; let $C=\Cst(d,b)$ be the constant 
considered in Proposition \ref{resistance-bound}. 
\begin{definition}
The sets $R_{xy}$ in our induction are classified as follows:
\begin{enumerate} 
\item[$\bullet$]  Class 1: diamonds $R_{xy}$, with $|x-y|>a^{1/4}$;
\item[$\bullet$] Class 2: deformed diamonds $R_{xy}$, with $|x-y|>a^{1/4}$;
\item[$\bullet$]  Class 3: deformed diamonds $R_{xy}$, with $|x-y|\le a^{1/4}$,
\end{enumerate}
{where $a=\inf a_{ij}$ was defined in \eqref{adef}.} 
\end{definition}
Our goal in this section is to prove the following theorem.
 \begin{theorem}
\label{thm:induction}
Let $m=a^{1/8}$, let $\rho=1/2$ and assume $a\ge a_0$ for some constant $a_0\ge1$. For all $n_1$, $n_2$ and $n_3\ge0$,  let $R_{x_iy_i}$, $i=1,\ldots n_1$, $R_{p_jq_j}$, $j=1,\ldots n_2$, and $R_{r_ks_k}$, $k=1,\ldots n_3$ be respectively subsets of class 1, 2 and 3. Then we have  
$$\lan\prod_{i=1}^{n_1}B_{x_iy_i}^m\prod_{j=1}^{n_2}B_{p_jq_j}^{3m}\ochi_{p_jq_j}\prod_{k=1}^{n_3}B_{r_ks_k}^{3m}\ran
\le2^{n_1}(1+\rho)^{n_2}2^{n_3}.$$
\end{theorem}

The proof is by induction  on $\max_{1\le i\le n_1}|x_i-y_i|$. Let ${\bf(H)_\ell}$ be the following statement: 
Theorem \ref{thm:induction} holds if $$\max_{1\le i\le n_1}|x_i-y_i|\le\ell.$$ 

The first step in the induction is the case $n_1=0$, proved in Section \ref{short} below: this will require $a\ge\Cst(d,b)$.

Let us now show that ${\bf(H)_{\ell-1}}$ implies ${\bf(H)_\ell}$. We do the proof only in the case where $n_1=1$, the general case being only notationally more involved.  

Assume ${\bf(H)_{\ell-1}}$. Let $x$, $y$ $\in\La$ be such that $\ell-1<|x-y|\le\ell$. 
Let $\tR_{xy}$ be the deformed diamond between $x$ and $y$ introduced in Definition \ref{deformed}, with $l=y-x$:  
this corresponds to a perfect diamond with angle
$\pi/16$ instead of $\pi/4$. Let $\tR_{xy}^x$ and $\tR_{xy}^y$ be its two parts in Definition \ref{deformedf}, 
respectively from $x$ and $y$, with $f_x=f_y=3/5$.   

Define
$$u_{xy}=\prod_{j\in\tR_{xy}^x}\chi_{xj}\prod_{j\in\tR_{xy}^y}\chi_{yj},$$
and let 
$$\Rc(x,y)=\sum_{z\in\tR_{xy}^x}B_{xy}^m\chi_{xz}^c\prod_{j:\,|j-x|<|z-x|}\chi_{xj}.$$
Then it follows from the expansion of the partition of the unity
$$1=\prod_{j\in\tR_{xy}^x}(\chi_{xj}+\chi_{xj}^c)\prod_{j\in\tR_{xy}^y}(\chi_{yj}+\chi_{yj}^c)$$
that
\begin{equation}
\label{part}
\lan B_{xy}^m\ran\le\lan B_{xy}^mu_{xy}\ran+\lan\Rc(x,y)\ran+\lan\Rc(y,x)\ran.
\end{equation}
The first term in the right-hand side of \eqref{part} can be upper bounded, by \eqref{protect3} in Lemma \ref{wardp}: if $a\ge\Cst(C)=\Cst(d,b)$ (recall $m=a^{1/8}$), then 
\begin{equation}
\label{1pr}
\lan B_{xy}^mu_{xy}\ran\le(1-mC/a)^{-1}\le1+\rho\le3/2.
\end{equation}
It remains to upper bound $\lan\Rc(x,y)\ran$. 

Now, $\Rc(x,y)$ being an expansion over ``bad'' points $z$ (i.e. sites $z$ such that $\chi_{xz}^c$ holds),  we expand it into four terms. 

{\bf (i)} First, the sites $z$ close $x$, i.e. with $|z-x|\le a^{1/4}$:
$$\Rc_1(x,y)=\sum_{z\in\tR_{xy}^x:\,|z-x|\le a^{1/4}}B_{xy}^m\chi_{xz}^c\prod_{j:\,|j-x|<|z-x|}\chi_{xj}.$$
We prove in Section \ref{sec:exp1} that $\lan\Rc_1(x,y)\ran\le1/16$ if $b\ge\Cst(d)$ and $\al m\ge d$ (Case 1 in \cite{dsz}).

{\bf (ii)} Second, fix a constant $M$, which will only depend on the dimension $d$. 
For a site $z$ far from $x$ (i.e. with $|z-x|>a^{1/4}$),  let
$$v_{x,y,z}=\prod_{\substack{j,k\in\tR_{xy}^x\cup\tR_{xy}^y,\, |j-z|\le|z-x|^{1/2},\\
M|z-x|^{1/2}\le|j-k|\le|z-x|/5}}\chi_{jk},$$
and let 
$$\Rc_2(x,y)=\sum_{z\in\tR_{xy}^x:\,|z-x|>a^{1/4}}B_{xy}^m\chi_{xz}^cv_{x,y,z}^c\prod_{j:\,|j-x|<|z-x|}\chi_{xj}.$$

Then $v_{x,y,z}^c=1$ if there is a large scale ``bad'' event originating from a point near $z$. 
We will show in Section \ref{sec:exp2} that the corresponding term $\lan \Rc_2(x,y)\ran\le1/16$ 
provided $b\ge\Cst(d)$ and $\al m\ge 10d$ (Case 2a in \cite{dsz}). 

{\bf (iii)} Third, we consider the case where $|z-x|>a^{1/4}$ and $v_{x,y,z}$ holds, 
i.e. there is no large scale ``bad'' event near $z$ and, furthermore, there is  a   point $g$ with 
$|g-z|\le|z-x|^{1/2}$ that is good up to distance $M|z-x|^{1/2}$. 

More precisely, given $i\in\La$, $R>0$, let   
$$G(i,R)=\prod_{h:\,|i-h|\le R}\chi_{ih};$$
then $G(i,R)=1$ iff $i$ is ``good'' up to distance $R$. Recall the similar Definition \ref{def:good} 
that a site $x\in\La$ is called $n$-good if $\chi_{xy}=1$, for all $y$ with $|y-x|\le4^n$.  

Let 
$$g_{x,y,z}=\max_{g:\,|g-z|\le|z-x|^{1/2}}G(g,M|z-x|^{1/2}).$$
Then $g_{x,y,z}=1$ iff we can find a site $g$ in the ball of radius $|z-x|^{1/2}$ centered at $z$, 
which is good up to distance $M|z-x|^{1/2}$. If moreover $v_{x,y,z}=1$ then $g$ is good up to
distance $|z-x|/5$. 

Hence we will see that, if $v_{x,y,z}=g_{x,y,z}=1$, then we can find a deformed diamond from $x$ to $g$ close to $z$ such that $\ochi _{x,g}=1$, so that we can apply the induction hypothesis. If we let
$$\Rc_3(x,y)=\sum_{z\in\tR_{xy}^x:\,|z-x|>a^{1/4}}B_{xy}^m\chi_{xz}^cv_{x,y,z}g_{x,y,z}\prod_{j:\,|j-x|<|z-x|}\chi_{xj},$$
then $\lan \Rc_3(x,y)\ran\le1/16$ provided $a^\al\ge\Cst(d)$ and $\al m\ge3d$ (Case 2b in \cite{dsz}): 
this is proved in Section \ref{sec:exp3}.
 
{\bf (iv)} Fourth, $|z-x|>a^{1/4}$, $v_{x,y,z}=1$ but $g_{x,y,z}=0$ then there is no good point up to distance 
$M|z-x|^{1/2}$  in the ball centered at $z$ of radius $|z-x|^{1/2}$. 
For $M$ small enough $M\le\Cst(d)$
this ball contains the hypercube centered at $z$ of side  $4M|z-x|^{1/2}$. 
This implies that  $\chi_{R_{n(x,z)}(z)}^c$ holds, where $n(x,z)$ is the  upper integer part 
of $\log (M|z-x|^{1/2})/\log 4$. 
Then, if we let 
$$\Rc_4(x,y)=\sum_{z\in\tR_{xy}^x:\,|z-x|>a^{1/4}}B_{xy}^m\chi_{xz}^cv_{x,y,z}\chi_{R_{n(x,z)}(z)}^c\prod_{j:\,|j-x|<|z-x|}\chi_{xj},$$
 we apply the ``good points'' expansion obtained in Section \ref{good} to show in Section \ref{sec:exp4} that $\lan \Rc_4(x,y)\ran\le1/16$ if $a\ge\Cst(d)$, $b\ge\Cst(d)$ and $\al m\ge8d$ (Case 2c in \cite{dsz}).

In summary, 
$$\lan \Rc(x,y)\ran\le\sum_{i=1}^4\lan \Rc_i(x,y)\ran,$$
and $\lan \Rc_i(x,y)\ran\le1/16$ is proved in Sections \ref{sec:exp1}, \ref{sec:exp2}, \ref{sec:exp3} and \ref{sec:exp4}, respectively in the cases $i=1$, $2$, $3$ and $4$. 

Let us summarize our choice of constants: $M\le\Cst(d)$ in {\bf(iv)}, $b\ge\Cst(d)$ in {\bf(i)-(ii)} and {\bf(iv)}, $a\ge\Cst(d,b)$ for the short-range fluctuations ($n_1=0$, Section \ref{short}) and in {\bf(iv)}, $\al\ge\Cst(d)/\log a$ in {\bf(iii)} (which implies in particular $\al m=\al a^{1/8}\ge10d$ required in {\bf(ii)-(iv)} if $a\ge\Cst$), and $m=a^{1/8}$ and $\rho=1/2$ in the statement of Theorem \ref{thm:induction}.
\subsection{Proof of Theorem \ref{thm:induction} in the case $n_1=0$.}
\label{short}
The proof is similar to the one of lemma 8 in \cite{dsz}. First, note that the case $n_3=0$ follows from 
Lemma \ref{wardp}, for $a\ge\Cst(C)$.

Let us do the proof in the case $n_2=0$ and $n_3=1$, the argument for the general case being only notationally more involved. 

Let $\de>0$. For all $p, q\in\La$, $p\sim q$, let 
$$\xi_{pq}=\1_{\{B_{pq}\le1+\de\}}.$$

Using $1\le \prod_{\{p,q\}, p, q\in R_{xy}}\xi_{pq}+\sum_{\{p,q\}, p, q\in R_{xy}}\xi_{pq}^c$, we have
\begin{equation}
\label{shortdec}
\lan B_{xy}^{3m}\ran\le\lan B_{xy}^{3m}\prod_{\{p,q\}, p, q\in R_{xy}}\xi_{pq}\ran+\sum_{\{p,q\}, p, q\in R_{xy}}\lan B_{xy}^{3m}\xi_{pq}^c\ran.
\end{equation}

Let us first deal with the first term in the right-hand side of \eqref{shortdec}: $B_{pq}\le1+\de$ implies 
\[
\delta\geq \cosh(U_p-U_q)-1\geq (U_p-U_q)^2/2\geq 0.
\]
Choose $\de>0$ such that $a^{1/4}\sqrt{2d\de}=1$, i.e. $\de=a^{-1/2}/(2d)$.

Let $z=x$, $y$, and assume $\prod_{\{p,q\}, p, q\in R_{xy}}\xi_{pq}$ holds. Let $|.|_1$ denote the $L^1$ norm on $\R^d$. 

Then, for all  $j\in R_{xy}$,  using Cauchy-Schwarz inequality, 
\begin{equation}
\label{Bzj1}
|U_z-U_j|\le|z-j|_1\sqrt{2\de}\le |z-j|\sqrt{d}\sqrt{2\de}\leq a^{1/4}\sqrt{2d\de}=1,
\end{equation} 
Subsequently, for all $p$, $q$ $\in R_{xy}$, $p\sim q$, 
\[
\delta \geq  B_{pq}-1\geq \frac{e^{U_p+U_q}}{2}(S_p-S_q)^2\geq \frac{e^{-2}}{2}e^{2U_z}(S_p-S_q)^2.
\]
which implies $e^{U_z}|S_p-S_q|\le\sqrt{2\de}e$. Using again our choice of $\de$, we deduce that,  again if $j\in R_{xy}$,
$e^{U_z}|S_z-S_j|\le a^{1/4}\sqrt{2d\de}e=e$, so that 
\begin{equation}
\label{Bzj2}
\frac{e^{U_j+U_z}}{2}(S_j-S_z)^2\le\frac{e}{2}e^{2U_z}(S_j-S_z)^2\le\frac{e^3}{2}.
\end{equation}
Inequalities \eqref{Bzj1}-\eqref{Bzj2} together imply that $B_{zj}\le\cosh(1)+e^3/2=\Cst$.

Therefore $\ochi_{xy}$ holds if  $b\ge\Cst$ and $\al=0$ implies
\begin{equation}
\label{1sh}
\lan B_{xy}^{3m}\prod_{\{p,q\}, p, q\in R_{xy}}\xi_{pq}\ran\le(1-3mC/a)^{-1}\le3/2,
\end{equation}
assuming $a\ge\Cst(C)=\Cst(d,b)$ (recall $m=a^{1/8}$).

Let us now deal with the second term in the right-hand side of \eqref{shortdec}: fix $p$, $q$ $\in R_{xy}$, $p\sim q$, and use that, by Markov inequality,
\[
\xi^{c}_{pq}\le\left(\frac{B_{pq}}{1+\delta }\right)^{a/2}.
\]
 Let $n=|x-y|_1$ and  $\ell-1< |x-y|\leq \ell$.  Then $n\le  \sqrt{d} |x-y|\leq \sqrt{d}\ell$, again by Cauchy-Schwarz inequality.

Let $(x_0,\ldots,x_n)$ be a path of minimal  $L^1$ distance from $x$ to $y$ inside $R_{xy}$, which does not go through the edge $\{p,q\}$. 
By repeated application of Lemma \ref{Bxyz}, 
$$2B_{xy}\le \prod_{0\le j\le n-1}2B_{x_jx_{j+1}}.$$
Therefore
\begin{align}
\nonumber
\lan B_{xy}^{3m}\xi_{pq}^c\ran&\le\frac{2^{3m\sqrt{d}(n-1)}}{(1+\de)^{a/2}}\lan B_{pq}^{a/2}\prod_{0\le j\le n-1}B_{x_jx_{j+1}}^{3m}\ran\le \frac{2^{3m\sqrt{d}(n-1)}}{(1+\de)^{a/2}}2\left(1-\frac{3m}{a}\right)^{-n}\\
\label{2sh}
&\le\exp\left(3m\sqrt{d}\ell-a\de/3\right)\le\exp\left(3a^{3/8}\sqrt{d}-a^{1/2}/(6d)\right).
\end{align}
In the second inequality, we use \eqref{protect1} and note that, if $r\sim s$, then $D_{rs}^N=a^{-1}_{rs}\leq  a^{-1}$. 
In the third inequality we assume $a\ge\Cst$ and use, for $x\in[0,1/2]$,  $(1-x)^{-1}\le e^{2x}$ and $(1+x)^{-1}\le e^{-2x/3}$. 
Finally, we use $\de=a^{-1/2}/(2d)$, $m=a^{1/8}$ and $\ell\le a^{1/4}$ (since $R_{xy}$ is of Class 3) in the last inequality. 

In summary, \eqref{shortdec}, \eqref{1sh} and \eqref{2sh} together imply, if $a\ge\Cst(C,d)=\Cst(d,b)$,  then
$$\lan B_{xy}^{3m}\ran\le3/2+\Cst(d)\ell^d\exp\left(3a^{3/8}\sqrt{d}-a^{1/2}/(6d)\right)\le2.$$

\subsection{Proof of $\lan \Rc_1(x,y)\ran\le1/16$}
\label{sec:exp1}
Let $z\in\tR_{xy}^x$ be such that $|z-x|\le a^{1/4}$. 
Using $\chi_{xz}^c\le B_{xz}^{2m}b^{-2m}|z-x|^{-2\al m}$ and Lemma \ref{Bxyz}, we obtain
$$B_{xy}^m\chi_{xz}^c\le2^mB_{xz}^mB_{zy}^m\chi_{xz}^c\le2^mb^{-2m}|z-x|^{-2\al m}B_{xz}^{3m}B_{zy}^m.$$

In order to apply the induction assumption, we would need to construct a deformed diamond $R_{xz}$ and diamond $R_{zy}$ which do not intersect within $R_{xy}$. This is not true in general, but we can add an intermediate point $a\in R_{xy}$ such that  $R_{xz}$, $R_{za}$ and $R_{xz}$ are respectively one deformed diamond and two diamonds within $R_{xy}$ and disjoint, except at endpoints (see Figure \ref{fig:gbound}, and Lemma 12 [1] in \cite{dsz} for more details). Now, using again Lemma \ref{Bxyz},
\begin{figure}
\centerline{\psfig{figure=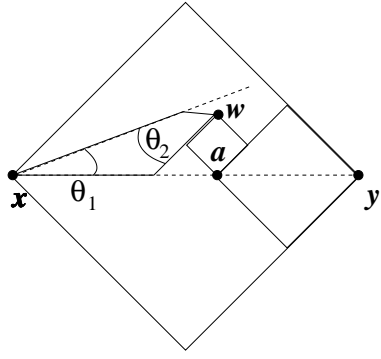, width=5cm}}
    \caption{(Figure 5 in \cite{dsz}) We add one intermediate point $a$. The two angles
    $\theta_{1}$ and $\theta_{2}$ are greater than $\pi/8$.}
\label{fig:gbound}
\end{figure}
$$B_{xy}^m\chi_{xz}^c\le2^{2m}b^{-2m}|z-x|^{-2\al m}B_{xz}^{3m}B_{za}^mB_{ay}^m.$$
Therefore, using the induction assumption, 
\begin{align*}
\lan\Rc_1(x,y)\ran\le\sum_{z\in\tR_{xy}^x:\,|z-x|\le a^{1/4}}4^{m+3}b^{-2m}|z-x|^{-2\al m}
\le\left(\frac{4}{b^2}\right)^m\Cst(d)\sum_{r=1}^{a^{1/4}}r^{d-1-2\al m}\le1/16
\end{align*}
if $b\ge\Cst(d)$ and $\al m\ge d$.
\subsection{Proof of $\lan \Rc_2(x,y)\ran\le1/16$}
\label{sec:exp2}
Let $z\in\tR_{xy}^x$ be such that $|z-x|>a^{1/4}$, and let $j$, $k$ $\in\tR_{xy}^x\cup\tR_{xy}^y$ such that
\begin{equation}
\label{largf}
|j-z|\le|z-x|^{1/2}\tx{ and }M|z-x|^{1/2}\le|j-k|\le|z-x|/5.
\end{equation}
As above, we use 
\begin{equation}
\label{chijk}
\chi_{jk}^c\le B_{jk}^{m}b^{-m}|j-k|^{-\al m}.
\end{equation} In order to apply the induction assumption, we need to expand $B_{jk}^{m}B_{xy}^m$ into a product of terms arising from disjoint diamonds within $R_{xy}$. 
It is an easy geometric result to show that, under our assumptions on $z$, $j$ and $k$, we can choose four intermediate points $a_i\in R_{xy}$, so that $R_{xa_1}$, $R_{a_ia_{i+1}}$ ($i=1,\ldots,3$) and $R_{a_4y}$ are diamonds with disjoint interiors which do not overlap with the diamond $R_{jk}$ (see Figure \ref{fig:gbound1}, and Lemma 12 [2] in \cite{dsz} for more details).
\begin{figure}
\centerline{\psfig{figure=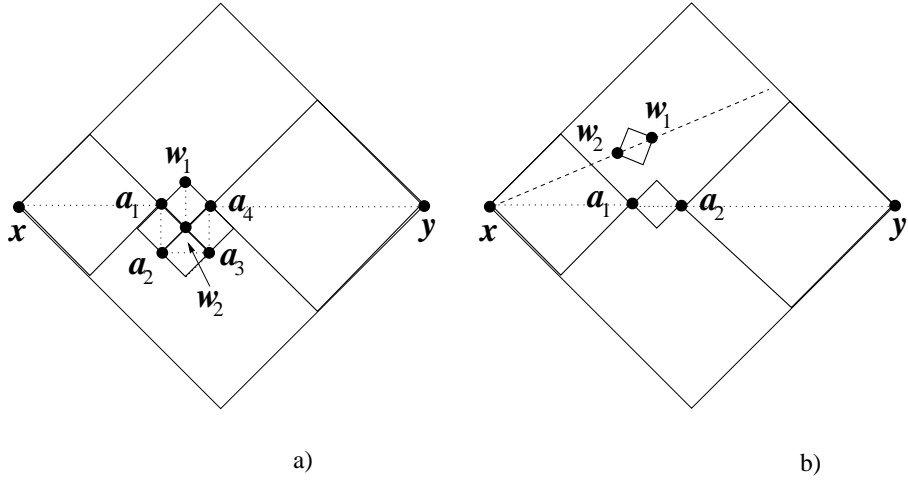, width=12cm}}
    \caption{(Figure 6 in \cite{dsz})
    a) If the pair $w_1 w_2$ is right in the middle, then
    we need to add four intermediate points $a_1, \ldots, a_4\,$ in order
    to find a minimal connected path around $w_{1}w_{2}$
    paved with disjoint diamonds.
    b) Even if the pair $w_1 w_2$ is located on the boundary of
    $\tilde{R}_{xy}^{x}\,$, the region $R_{w_1 w_2}$ still lies
    inside $R_{xy}\,$.} \label{fig:gbound1}
\end{figure}
Now, using \eqref{chijk} and Lemma \ref{Bxyz}, 
$$B_{xy}^m\chi_{jk}^c\le 2^{4m}B_{xa_1}^m\prod_{i=1}^3B_{a_ia_{i+1}}^mB_{a_4y}^mB_{jk}^mb^{-m}|j-k|^{-\al m},$$
which implies, using the induction assumption  and $|j-k|\geq M|z-x|^{1/2}$, 
$$\lan B_{xy}^m\chi_{jk}^c\ran\le 2^{4m}2^6b^{-m}|z-x|^{-\al m/2}M^{-m\alpha }.$$
Now, for each $z\in\tR_{xy}^x$, there are of the order of $\Cst(d)|z-x|^{d+d/2}$ pairs $(j,k)$ satisfying \eqref{largf}. Therefore
\[
\lan\Rc_2(x,y)\ran\le\left(\frac{16}{bM^{\alpha }}\right)^m\Cst(d)\sum_{r>a^{1/4}}r^{(d-1)+d+d/2-\al m/2}\le\frac{1}{16}
\]
if  $b\ge\Cst(d)$ (recall that $M$ depends only on $d$) and $\al m\ge 10d$.  
\subsection{Proof of $\lan \Rc_3(x,y)\ran\le1/16$}
\label{sec:exp3}
Let $z\in\tR_{xy}^x$ be such that $|z-x|>a^{1/4}$. 

If $v_{x,y,z}=g_{x,y,z}=1$, then there exists $g$ with $|g-z|\le|z-x|^{1/2}$ such that, for all $h\in R_{xy}$ with $|g-h|\le |z-x|/5$, $\chi_{gh}$ holds. Let $R_{xg}$ be the deformed diamond in Definition \ref{deformed}, with $l=g-x$, and choose 
$f_x=\min\{1,(|z-x|-1)/|g-x|\}\geq 1-O(a^{-1/8})$,
$f_{g}=1/5$ in Definition \ref{deformedf}. Then $\ochi_{xg}$ holds.

On the other hand, if $\chi_{xz}^c$ occurs then, using $B_{gz}\le b|z-g|^\al\le b|z-x|^{\al/2}$ and Lemma \ref{Bxyz}, we deduce that 
$2B_{xg}\ge B_{xz}/B_{gz}\ge b|z-x|^\al/(b|z-x|^{\al/2})=|z-x|^{\al/2}$.

Hence
$$B_{xy}^m\chi_{xz}^c\ochi_{xg}\le2^mB_{xg}^mB_{gy}^m\chi_{xz}^c\ochi_{xg}
\le2^{3m}(B_{xg}^{3m}\ochi_{xg})B_{gy}^m|z-x|^{-\al m}.$$

As in the proof of the case $i=1$, we introduce an intermediate point $a\in R_{xy}$ such that $R_{ga}$ and $R_{ay}$ are diamonds, 
disjoint from each other and from $R_{xg}$, except at endpoints (see Figure \ref{fig:gbound}, and Lemma 12 [1] in \cite{dsz} 
for more details).

Therefore
$$\lan B_{xy}^m\chi_{xz}^c\ochi_{xg}\ran\le 2^{4m}(1+\rho)2^2|z-x|^{-\al m}.$$
There are less than $|z-x|^{d/2}$ choices for $g$ given $z$, so that 
$$\lan\Rc_3(x,y)\ran\le2^{4m}\Cst(d)\sum_{r>a^{1/4}}r^{d-1+d/2-\al m}\le1/16$$
if $\al m\ge3d$ and $a^\al\ge\Cst(d)$. 
\subsection{Proof of $\lan \Rc_4(x,y)\ran\le1/16$}
\label{sec:exp4}
If $a\ge\Cst$, then $R_{n(x,z)}(z)$ is inside $R_{xy}$; recall that its side length is of the order of 
$|z-x|^{1/2} \ll|y-x|$. 
As in the proof of the case $i=2$, we can choose four intermediate points $a_i\in R_{xy}$, so that $R_{xa_1}$, 
$R_{a_ia_{i+1}}$ ($i=1,\ldots,3$) and $R_{a_4y}$ are diamonds with disjoint interiors which do not overlap with 
the hypercube $R_{n(x,z)}(z)$ (see Figure \ref{fig:gbound1}, and Lemma 12 [2] in \cite{dsz} for more details).

Now, using Lemma \ref{Rn}, 
\[
B_{xy}^m\chi_{R_{n(x,z)}(z)}^c\le 2^{4m}B_{xa_1}^m\prod_{i=1}^3B_{a_ia_{i+1}}^mB_{a_4y}^m
\sum_{T \in \mathcal{T}_n}  \prod_{v \in L_T} 
\mathop{\sum_{x_{v} \in R_v,\, y_{v} \in \Lambda}}_{4^{n-k_{v}-1} < |x_{v} - y_{v}| \leq 4^{n-k_{v}}}
 \frac{ B_{x_{v}y_{v}}^m }{b^m |x_{v}-y_{v}|^{\alpha m}}.
\]
This implies, letting $n_v=n-d_v$ and using the induction assumption, 
$$\lan B_{xy}^m\chi_{R_{n(x,z)}(z)}^c\ran\le2^{4m}2^5\sum_{T \in \mathcal{T}_n} \prod_{v \in L_T}
  \frac{ (4^{n_{v}})^d (4^{n_{v}} 2)^d }{b^m 4^{(n_{v}-1)\alpha m}}=
  2^{4m+5}I_{n(x,z)}$$
  where, for all $n\ge0$, 
  $$I_{n}=\sum_{T \in \mathcal{T}_n} \prod_{v \in L_T}\g2^{-\ze n_v},$$
  with $\g=(4^\al/b)^m2^d$, $\ze=2(\al m-2d)$.

It follows from the structure of the trees $\mathcal{T}_n$ that 
$$I_n=\g2^{-\ze n}+(I_{n-1})^{2^d},\,\,I_0=\g.$$
Assume $\al m\ge4d$ and $b\ge\Cst(d)$ (recall $\al\le1/4$), so that $\g\le1/4$ and $\ze\ge\al m$. Then we deduce by elementary induction that  $$I_n\le2^{-\al mn}\tx{ for all }n\ge1.$$ 
Note that $4^{n(x,z)}\ge2^M\sqrt{|z-x|}$. In summary, 
$$\lan \Rc_4(x,y)\ran\le2^{4m}2^{-\al mM/2}\Cst(d)\sum_{r>a^{1/4}}r^{d-1-\al m/4}\le1/16$$
if $\al m\ge8d$ and $a\ge\Cst(d)$.

\section{Proof of Ward inequalities: Lemmas \ref{ward} and \ref{wardp}} 
\label{ward-resist}
We start this section with two lemmas:
the first is an elementary lemma which expresses the equivalent resistance on a conductance network as a quadratic
form and relates the quantity $D_{x,y}$ to the corresponding term in \cite{dsz}.
\begin{lemma}\label{electric}
Let $(V, E)$ be a finite connected graph and $(c_e)_{e\in E}$ a conductance network on $E$. 
We set $c_i=\sum_{j\sim i} c_{i,j}$.
Let $i_0\in V$ be a fixed vertex and $M$ be the matrix given by 
$$
 (M_{i,j})= \left\{\begin{array}{ll} {-c_{i,j}}, &\hbox{ $i\neq j$}
 \\
 c_i , &\hbox{ $i=j$}
\end{array}\right.
$$
(which is the matrix of the generator of the Markov process with jump rates $(c_{i,j})$).
Let $N$ be the restriction of $M$ to $V\setminus\{i_0\}$.
Denote by $G$ the $V\times V$ symmetric matrix defined by $G(i_0,y)=G(y,i_0)=0$ for any $y$ and
$G(x,y)=N^{-1}_{x,y}$ if $x,y\neq i_0$. 
If $D_{x,y}$ is the equivalent resistance between $x$ and $y$, then
$$
D_{x,y}=G(x,x)-2G(x,y)+G(y,y)= <(\delta_x-\delta_y),G (\delta_x-\delta_y)>.
$$
\end{lemma}

\begin{remark}
In comparison with \cite{dsz}, it means that the term $G_{x,y}$ which appears in formula (5.4) in \cite{dsz} is
the equivalent resistance between $x$ and $y$ with conductances $c^{x,y}_{i,j}= {e^{-t_x-t_y} B_{x,y}} \beta_{i,j}e^{t_i+t_j}$. In our context $i_{0}=0$ and $V=\Lambda$.
\end{remark}

 \begin{proof}
We first interpret probabilistically the matrix $G$.
Let $\ggg_{i_0}(x,y)$ be the  average number of visits on $y$ for the Markov chain 
 starting at $x$
with transition probabilities $p_{i,j}=\frac{c_{i,j}}{c_i}$ and killed at its first
entrance hitting time of $i_0$
$$
\ggg_{i_0}(x,y)= \E_{x}\left( \sum_{k=0}^{H_{i_0}} \indic_{\{X_k=y\}}\right),
$$
where $H_{i_0}= \inf\{k\ge 0,\; X_k =i_0\}$. Then,   by 
\cite[Chapt.2: ex.2.60,(a)]{lyons-peres}, we have 
$$
G(x,y)= \frac{1}{c_x} \ggg_{i_0}(x,y),
$$
and  item (d) in the same exercise yields the result.
 \end{proof}
 
The next lemma ensures that the joint  density $\mu_{a,\Lambda } (u,s)$ given in \eqref{usmes} 
has bounded ``moments'' up to a certain order. 
\begin{lemma}\label{moments}
Let $e_{1},\dotsc e_{n}$ be $n$ undirected edges in $E$.
Then 
\[
\lan\prod_{j=1}^{n} B_{e_{j}}^{m_{j}} \ran \leq    2^{n}
\]
for any choice of $m_{1},\dotsc m_{n}$ such that $m_{j}\leq a/2$ for all $j=1,..,n$, { where 
$a=\inf a_{ij}$ was defined in \eqref{adef}}.
\end{lemma}
\begin{proof}
By expression  \eqref{usmes}, we have
\begin{equation*}
\lan\prod_{j=1}^{n} B_{e_{j}}^{m_{j}} \ran
=\tfrac{1}{(2\pi)^{(N-1)}}\int \left[\prod_{e} \frac{1}{B_e^{\bar{a}_e}}\ \right]
 D[M_{a}(u,s)] e^{-\sum_{j\in\La}u_{j}}\prod_{k\ne 0} du_k ds_k,
\end{equation*}
where we set $\bar{a}_{{e_{j}}}= a_{e_{j}}-m_{j}$ for $j=1,..n$ and $\bar{a}_{e}=a_{e}$
for all other edges. Note that $\bar{a}_{{e_{j}}}\geq a_{e_{j}}/2$ since $m_{j}\leq a/2$ for all $j$.
Expanding the minor as  a sum over spanning trees we deduce (recall $c_{ij}=a_{ij}e^{u_i+u_j}/B_{ij}$)
\begin{align*}
&\lan\prod_{j=1}^{n} B_{e_{j}}^{m_{j}} \ran
=\tfrac{1}{(2\pi)^{(N-1)}}  \sum_{T} \int \left[\prod_{e} \frac{1}{B_e^{\bar{a}_e}}\ \right]
 \left[ \prod_{e\in T} c_{e}\right] e^{-\sum_{j\in\La}u_{j}}\prod_{k\ne 0} du_k ds_k,\cr
& \leq 2^{n} \ \tfrac{1}{(2\pi)^{(N-1)}}  \int \left[\prod_{e} \frac{1}{B_e^{\bar{a}_e}}\ \right]
 D[M_{\bar{a}}(u,s)] e^{-\sum_{j\in\La}u_{j}}\prod_{k\ne 0} du_k ds_k
\ = 2^{n} \int d\mu_{\bar{a},\Lambda } (u,s)= 2^{n}
\end{align*}
where we have used the bound $a_{e_{j}}\leq 2 (a_{e_{j}}-m_{j})$ so we can replace $a_{e}$ with  $\bar{a}$ 
in the determinant. 
\end{proof}

 {\bf Proof of Lemma \ref{ward}.} For more readability we provide an elementary derivation of the Ward 
identity which does not involve
 fermionic  integral, even though it could be deduced from the more general proof of Lemma  \ref{wardp}. 
 Consider the graph $(\Lambda , \tilde E)$ where we add an extra edge $\tilde e=\{x,y\}$ to $E$ 
(possibly creating a double edge).
 We put a weight $a_{\tilde e}=-m$ on this edge. 
Denote by $\tilde \mu_{a,\Lambda }(u,s)$  the corresponding density, and by $\tilde{M}_{a}$ the corresponding matrix 
in (\ref{usold}).   If $m<0$, it corresponds to the density associated with the new graph $(\Lambda , \tilde E)$ 
by theorem \ref{meas}
and so  $\int d\tilde \mu_{a,\Lambda }=1$. 

The strategy is now to extend the equality to $0\le m\le a/4$
 by analyticity. For this we need to upper bound the density $\tilde \mu_{a,\Lambda }$.
Using the expression \eqref{usmes}, we deduce, for $m\ge0$, 
\begin{align*}
| \tilde \mu_{a,\Lambda }(u,s)|&\leq   \tfrac{1}{(2\pi)^{(N-1)} } e^{-\sum_{j\in\La}u_{j}}
  \left[\prod_{e\in E} \frac{1}{B_e^{a_e}}\ \right] B_{xy}^{m}\  
D[\tilde{M}_{|a|}]\\
&=  B_{xy}^{2m} \tfrac{1}{(2\pi)^{(N-1)}} e^{-\sum_{j\in\La}u_{j}}
\left[\prod_{e\in \tilde{E}} \frac{1}{B_e^{|a_e|}}\ \right] D[\tilde{M}_{|a|}]
=   B_{xy}^{2m}  \tilde \mu_{|a|,\Lambda }(u,s).
\end{align*}
Now let $\gamma $ be a simple 
path in $E$ connecting $x$ to $y$.
Then, by Lemma \ref{Bxyz}, 
\[
 B_{xy}^{m}\leq 2^{m (|\gamma |-1)} \prod_{e\in \gamma } B_{e}^{m},
\]
and, by  the same argument as in the proof of Lemma \ref{moments} above, if $0\le m\le a/4$, then
\[
\int |d\tilde \mu_{a,\Lambda }(u,s)| \leq  4^{m (|\gamma |-1)} \int 
\left[ \prod_{e\in \gamma } B_{e}^{2m}\right]  
d\tilde \mu_{|a|,\Lambda }(u,s)
\leq   4^{m (|\gamma |-1)}  2^{|\gamma |} \int  d\tilde \mu_{\bar{a},\Lambda }(u,s)
\]
where  $\bar{a}_{e}=|a_{e}|-2m=a_{e}-2m{\geq a/2}$ for all $e\in \gamma $ 
(as $e\ne\{x,y\}$ { and $m\leq a/4$}) and $\bar{a}_{e}=|a_{e}|$ otherwise. 
By Proposition \ref{newm}, $d\tilde \mu_{\bar{a},\La}(u,s)$ is a probability measure on $\tilde{E}$. 
The bound above holds for any 
$m\leq a/4$, hence $\tilde \mu_{a,\La}(u,s)$ is integrable. Moreover it is an analytic function
in the parameter $m$, and  therefore
\begin{equation}\label{integ}
\int \tilde \mu_{a,\Lambda }(u,s)  \prod_{k\ne 0} du_k ds_k= 1, \;\;\; {\forall \; m\le a/4}.
 \end{equation}
Now, {as in Lemma \ref{electric},} 
let  $N=N(u,s)$  (respectively $\tilde N=\tilde N(u,s)$) be the restriction of $M_{a}(u,s)$
 (resp. $\tilde M_{a}(u,s)$) to the subset 
 $\Lambda \setminus\{0\}$.
 Expanding the determinant with respect to the extra term coming from the new edge 
$\tilde e$ { and using Cramer's rule} we deduce 
\begin{align*}
 &  D[\tilde M_{a}(u,s)]= \det(\tilde N(u,s))\\
&=\det(N) 
  + (-m)\frac{e^{u_x+u_y}}{ B_{x,y}} \left[ \indic_{x\neq 0} \det(N)_{x,x} -
 2 \indic_{x\neq i_0, y\neq 0} \det(N)_{x,y} +\indic_{y\neq 0} \det(N)_{y,y} \right]
 \\
 &=
 \det(N) \left[ 1-m \frac{e^{u_x+u_y}}{B_{x,y}}<(\delta_x-\delta_y), G (\delta_x-\delta_y)> \right]\\
&  = D[M_{a}(u,s)] \left[ 1-m <(\delta_x-\delta_y), G^{xy} (\delta_x-\delta_y)> \right] ,
\end{align*}
 where  $\det(N)_{x,y}$ is the cofactor $(x,y)$ of $N$ 
that coincides (up to a sign) with 
 the minor obtained by removing the line $x$ and column $y$. In the  
  last lines $G$ and $G^{xy}$ are the matrices defined in Lemma \ref{electric},  
with conductances $c_{i,j}$ and $c_{ij}^{xy}$ defined in Proposition 
\ref{newm} and Definition \ref{def:er}.  Note that for all $x\neq y$  
\[
<(\delta_x-\delta_y), G (\delta_x-\delta_y)>=\indic_{x\neq 0, y\neq 0} 
<(\delta_x-\delta_y), N^{-1} (\delta_x-\delta_y)> +
\indic_{y\neq 0,x=0} N^{-1}_{yy} +\indic_{x\neq 0,y=0} N^{-1}_{xx}. 
\]
Finally, using Lemma \ref{electric} and  Definition \ref{def:er}, we conclude that
 $$
  \det(\tilde N(u,s))= \det(N(u,s))(1-mD_{x,y}).
  $$
  Therefore
 \begin{eqnarray*}
 1= \int \tilde \mu_{a,\Lambda }(u,s)  \prod_{k\ne 0} du_k ds_k= \int B^{m}_{x,y}(1-m D_{x,y}) \mu_{a,\Lambda }(u,s)  \prod_{k\ne 0} du_k ds_k.
 \end{eqnarray*}
\qed

 
{\bf Proof of Lemma \ref{wardp}.} 
In \cite{dsz}, the protected Ward estimates are a consequence of the Berezin identity stated in 
appendix C, proposition 2 of \cite{dsz}. 
The starting point is to write the  determinant term as a fermionic integral 
(cf e.g. \cite{Disertori-HDR,dsz}) with new pairs
of anticommuting variables $(\opsi_i, \psi_i)$:
$$D[M_{a}(u,s)]=(2\pi)^{N-1}\int \exp\left(-\sum_{e\in E}  \frac{a_{e}}{B_{e}} (S_{e}-B_{e}) \right)
\prod_{k\ne 0}d\opsi_k d\psi_k.$$  
This leads to the following density
\begin{equation}\label{usold2}
\mu_{a,\Lambda } (u,s, \opsi,\psi) = \left[\prod_{e\in E} \frac{1}{B_e^{a_e}}\ \right]
 e^{-\sum_{e\in E} \frac{a_{e}}{B_{e}} (S_{e}-B_{e})  }   e^{-\sum_{j\in\La}u_{j}},
\end{equation}
where 
$$
S_{i,j}= B_{ij}+ e^{u_i+u_j} (\opsi_i-\opsi_j)(\psi_i-\psi_j)
$$
is the same supersymmetric expression introduced in  \cite{dsz}, 
 $u_{0}=s_{0}=0$ and $\opsi_{0}=\psi_{0}=0$. Then 
$$
d\mu_{a,\Lambda } (u,s)=\int d\mu_{a,\Lambda } (u,s, \opsi,\psi) \prod_{k\ne 0} d\opsi_k d\psi_k.
$$
From the mathematical point of view, the fermionic integral should be understood as 
a product of derivatives   with 
respect to the variables 
$(\opsi_k, \psi_k)$, cf e.g. \cite{Disertori-HDR}. 
Since the fermionic variables are
antisymmetric, we have $(S_{e}-B_{e})^{2}=0$, and
\begin{equation}
\label{bs}
 e^{-\frac{a_{e}}{B_{e}} (S_{e}-B_{e}) }= 1- \frac{a_{e}}{B_{e}} (S_{e}-B_{e}) = 
\frac{1}{ 1+ \frac{a_{e}}{B_{e}} (S_{e}-B_{e}) }=\left(\frac{B_e}{S_e}\right)^{a_e}
\end{equation}
Then
\begin{equation}\label{usold2bis}
\mu_{a,\Lambda } (u,s, \opsi,\psi) = \left[\prod_{e\in E} \frac{1}{S_e^{a_e}}\ \right]
  e^{-\sum_{j\in\La}u_{j}}
\end{equation}
If $f(s_{ij})$ is a smooth function of the variable 
$(s_{i,j})$, we understand $f(S_{i,j})$ as the function
obtained by Taylor expansion in the fermionic variables 
$(\opsi_i,\psi_i,\opsi_j,\psi_j)$; 
this expansion is finite since the fermionic variables are
antisymmetric. Hence, formally $f(S_{i,j})$ is function with values in the exterior 
algebra 
constructed from $(\opsi_i, \psi_i,\opsi_j, \psi_j)$. 
The same can be generalized to $f[(S_{ij})_{i,j\in \Lambda \times \Lambda }]$.
The Berezin identity (see for instance proposition 2 in \cite{dsz}) implies that for any smooth function 
$f[(s_{ij})_{ij}]$, then
 \begin{equation}\label{Berezin}
 \int f[(S_{ij})_{i j}]d\mu_{a,\Lambda } (u,s, \opsi,\psi) = f(0),
\end{equation}
if $f[(s_{ij})_{i j}]$ is integrable with respect to $d\mu_{a,\Lambda } (u,s, \opsi,\psi)$
(which means that after integration with respect to the fermionic variables it is integrable 
in the usual sense). 

Remark that if $f$ is a polynomial (of degree bounded by $a$), then \eqref{Berezin} is a direct 
consequence of the fact that 
the measure $d\mu_{a,\Lambda } (u)$ has integral one, which can be proved either 
by supersymmetric or 
probabilistic arguments (as we did in \eqref{integ} above). 
Indeed let $(x_{i},y_{i})$, $i=1,\dotsc ,n$ be $n$ pairs of points.
As in the proof of Lemma \ref{ward} above let $\tilde{E}= E\cup (\cup_{i=1}^{n} \{{x_{i},y_{i}}\})$
the graph obtained by adding the edges  $\{x_{i},y_{i}\}$.  We assign to each new edge  $\{x_{i},y_{i}\}$
the conductance $a_{x_{i},y_{i}}=-m_{i}$. Let $\tilde{M}_{a}$ be the corresponding matrix.
Then, using \eqref{bs},
\[
\lan \prod_{i=1}^n S_{x_iy_i}^{m_i} \ran = 
\tfrac{1}{(2\pi)^{(N-1)}}\int \left[\prod_{e\in \tilde{E}} \frac{1}{B_e^{a_e}}\ \right]
 D[\tilde{M}_{a}(u,s)] e^{-\sum_{j\in\La}u_{j}}\prod_{k\ne 0} du_k ds_k,
\]
Similarly as in the proof of  Lemma \ref{ward}, we choose $n$ simple paths $\gamma_{i}$, $i=1,..n$ in $E$ 
connecting $x_{i}$ to $y_{i}$, and deduce that there exists $c>0$ such that  this density 
$\mu_{a,\Lambda } (u,s)$ is integrable if $m_i\le c$ for all $i=1,\dotsc ,n$. 
In general the constant $c$ depends on $n$, since an edge may belong to several
paths (the worst case being when it belongs to {\em all} paths) hence the corresponding
term may appear with a power as large as $n$.  
Below we will apply these relations to the case when the  regions 
$R_{x_iy_i}$ have disjoint interiors. In this last case we can always choose non overlapping
paths $\gamma_{i}$ and the constant $c$ will be independent of $n$. 

The density $\mu_{a,\Lambda } (u,s)$ is 
 analytic in the parameters $m_i$, and  \eqref{usold2bis} implies
\[
\lan \prod_{i=1}^n S_{x_iy_i}^{m_i} \ran = 1
\]
for negative $m_i$,  therefore  red this remains true when  
$0\leq m_i\le c$ for all $i$.  

Now by the same arguments as in \cite[lemma 7]{dsz} we have,  if the regions 
$R_{x_iy_i}$ have disjoint interiors,
\[
 D[\tilde{M}_{a}(u,s)] \geq  D[M_{a}(u,s)] \prod_{j=1}^{n} 
(1-m_{j} D^{N}_{x_{j}y_{j}}) .
\]
This completes the proof of \eqref{protect1}.

Finally to prove  the  ``protected'' Ward estimate \eqref{protect1} we proceed 
exactly as in  \cite[lemma 6]{dsz}.
The main idea is to approximate the characteristic function {
$\chi_{xy}= \indic_{\{B_{xy}\leq b |x-y|^{\alpha } \}}$ by a sequence of
smooth decreasing functions $\chi_{\delta}:\mathbb{R}\to\mathbb{R}$. } Now, by symmetry,
\[
\lan S_{xy}^{m} \chi_{\delta } (S_{xy})  \ran = 1 \qquad \forall x,y.
\]
Integration over the fermionic variables yields
\begin{align*}
&\lan S_{xy}^{m}\,\chi_{\delta } (S_{xy})  \ran= \lan B_{xy}^{m}\,
\chi_{\delta } (B_{xy})
 e^{\left[m+ \frac{\chi'_{\delta } (B_{xy})}{\chi_{\delta } (B_{xy})}\right]
\frac{(S_{xy}-B_{xy}) }{B_{xy}}}
  \ran \cr
&= \tfrac{1}{(2\pi)^{(N-1)}}\int \left[\prod_{e\in E} \frac{1}{B_e^{a_e}}\ \right] 
B_{xy}^{m}\,\chi_{\delta } (B_{xy})
 D[\tilde{M}_{a}(u,s)] e^{-\sum_{j\in\La}u_{j}}\prod_{k\ne 0} du_k ds_k,
\end{align*}
where $\tilde{M}_{a}$ is defined on the graph $\tilde{E}= E\cup (x,y)$ 
and the edge $(x,y)$
has conductance 
\[
 a_{xy}= \left[-m+ \frac{-\chi'_{\delta } (B_{xy})}{\chi_{\delta } (B_{xy})}  \right]\geq -m  ,
\]
since $\chi'<0$. Hence 
\[
 D[\tilde{M}_{a}(u,s)]>D[\tilde{M}_{\bar{a}}(u,s)]= D[M_{a}(u,s)] (1-mD_{xy})
\]
where $\bar{a}_{xy}=-m$ and $\bar{a}_{e}=a_{e}$ for all other $e$.
Finally
\[
\lan B^{m}_{xy}\chi_{\delta } (B_{xy}) (1-mD_{xy})  \ran \leq 
 \lan S_{xy}^{m}\,\chi_{\delta } (S_{xy})  \ran=  1.
\]
The proof of the general statement follows from a combination of this argument
with the ideas used for  the estimate \eqref{protect1}. \qed

\section{Proof of Proposition \ref{resistance-bound}}
\label{sec:res}
Denote by 
$$E_{x,y}=\{\{i,j\}, \;\; i\in R_{x,y}, \; j\in R_{x,y}, \; i\sim j\}
$$ 
the set of {non-directed} edges in $R_{x,y}$. We denote by $\tilde E_{x,y}$ the associated set of {\bf directed} edges.
Recall that $aD^N_{x,y}$ is the effective conductance
between $x$ and $y$ of the network with edges $E_{x,y}$ and conductances $(\gamma_{i,j})_{\{i,j\}\in E_{x,y}}$
 defined in \eqref{gammadef} of Lemma \ref{estcond}.
Let $\fff_{x,y}$ be the set of unit flows from $x$ to $y$ with support in $\tilde E_{x,y}$: 
precisely, $\theta\in \fff_{x,y}$
if $\theta$ is a function $\theta : \tilde E_{x,y}\rightarrow \R$ which is antisymetric 
(i.e. $\theta(i,j)=-\theta(j,i))$ and such that
$$
\dive(\theta) = \delta_x-\delta_y,
$$
where $\dive : R_{x,y} \rightarrow \R$ is the function
$$
\dive(\theta)(i)= \sum_{j\in R_{x,y}, \; j\sim i} \theta(i,j).
$$
Recall that (see for instance \cite{lyons-peres}, Chapter 2)
\begin{eqnarray}\label{energy}
a D_{x,y}^N = \inf_{\theta \in \fff_{x,y}} \sum_{\{i,j\}\in E_{x,y}} \frac{1}{ \gamma_{i,j}} (\theta(i,j))^2.
\end{eqnarray}
The strategy is now to construct explicitly a flow $\theta$  so that under the condition $\overline\chi_{x,y}$, 
the energy (\ref{energy}) is
bounded by a constant depending only on $d$, $\alpha$ and $b$. 
This flow will be constructed as an integral over flows associated
to sufficiently spreaded paths.  

Remind that a deformed diamond is a set of the following form
$$
\Z^d\cap(\tilde C_x^l \cap \tilde C_y^{x-y}),
$$
(plus a few points close to $x$ and to $y$ so that the set is connected in $\Z^d$) where
$x\in \Z^d$, $l\in \R^d$,  $l\neq 0$  and $y\in \Z^d$ is a point such that
$
y\in \tilde C^l_x.
$

For $z\in \R^d$ we set
$$
r(z)= \frac{(z-x)\cdot (y-x)}{ |y-x|^2},
$$
and $p(z)= x+r(z)(y-x)$ the projection of $z$ on the line $(x,y)$.

For $h\in (0,1)$ we denote 
$$
\hat R^x_{x,y}=\{i\in R_{x,y},\; r(i)\le h\},
\;\;\; \hat R^y_{x,y}=\{i\in R_{x,y},\; r(i)\ge h\}.
$$
From the assumption on $f_x, f_y$, there exists $h\in [1/10, 9/10]$ such that
$ \hat R^x_{x,y}\subset R^x_{x,y}$ and $\hat R^y_{x,y}\subset R^y_{x,y}$
 \footnote{Indeed, if $i\in R_{x,y}$, the angle $\angle (z,i),(x,y)\le \pi/8$ for $z=x,y$.
Hence, if $i\in  \hat R^x_{x,y}$, then $\vert i-x\vert { \leq \frac{r (i)}{\cos(\pi/8)}\vert x-y\vert} 
\leq  \frac{h}{\cos(\pi/8)}\vert x-y\vert$. 
Hence $\hat R^x_{x,y}\subset R_{x,y}^x$ as
soon as $h\le f_x \cos(\pi/8)$. But $\cos(\pi/8)\ge 0.92$ and $f_x\ge 1/5$ so that $f_x cos(\pi/8)\ge 0.18$.
Similarly  $\hat R^y_{x,y}\subset R_{x,y}^y$ as
soon as $1-h\le f_y \cos(\pi/8)$, since $f_y \cos(\pi/8)\ge 0.18$. Using $f_x+f_y\ge1+1/5$ it implies that
$(f_x+f_y) \cos(\pi/8) >1$ and we can find $h\in [1/10, 9/10]$ such that $\hat R^x_{x,y}\subset R_{x,y}^x$ 	and
$\hat R^y_{x,y}\subset R_{x,y}^y$.}.
We fix now such a $h\in [1/10, 9/10]$.
We set
$$
\Delta_h=\{ z\in \R^d, \; r(z)=h\} \cap \left( \tilde C_x^l \cap \tilde C_y^{x-y}\right).
$$
If is clear from the construction that there exists a constant $\Cst(d)$ such that 
\begin{equation}\label{deltah}
\vert \Delta_h\vert \geq \Cst(d) \vert x-y\vert^{d-1}, \qquad  \forall h\in   [1/10, 9/10], 
\end{equation}
where $\vert \Delta_h\vert$ is the surface of  $\Delta_h$.

To any path $\s=(x_0=x,\ldots,x_n=y)$ from $x$ to $y$ we can associate the unit
flow from $x$ to $y$ defined by
$$
\theta_\sigma=\sum_{i=1}^n \indic_{(x_{i-1},x_i)}-\indic_{(x_i,x_{i-1})}.
$$

For $u\in \Delta_h$, let $L_u$ be the union of segments
$$
L_u=[x,u]\cup[u,y].
$$
Clearly $L_u\subset \tilde C_x^l \cap \tilde C_y^{x-y}$ by convexity.
There is a constant $\Cst(d)$, such that for any $u\in \Delta_h$, we can find a simple path $\sigma_u$ in 
$R_{x,y}$ from $x$ to $y$ such that
for all $k=0, \ldots, \vert\sigma_u\vert$
\begin{eqnarray}\label{distance}
\hbox{dist}(L_u,\sigma_u(k))\le \Cst(d),
\end{eqnarray}
and we define $\theta$ as
$$
\theta = \frac{1}{ \vert \Delta_h\vert} \int_{\Delta_h} \theta_{\sigma_u}\,du,
$$
which is a unit flow from $x$ to $y$.

The path $\sigma_u$ can visit a vertex $i$ only if $\hbox{dist}(i, L_u)\le c_0=\Cst(d)$. This implies that, for all $i$, 
$$
\sum_{j:\,j\sim i} \vert \theta(i,j)\vert \leq
\frac{  2d }{\vert \Delta_h\vert} \int_{\Delta_h} \indic_{\{\hbox{dist}(i,L_u)\le c_0\}} du.$$
Now if $r(i)\in { (}0,h]$, {let $u_{0}=x+ \frac{h}{r (i)} (i-x)$ be the 
intersection  with $\Delta_{h}$ of the line passing through $i$ and $x$.} Then  
$$
\int_{\Delta_h} \indic_{\{\hbox{dist}(i,L_u)\le c_0\}} du 
\leq  
  { 
\int_{\Delta_h} \indic_{\{|u-u_{0}| \leq c_0\frac{h}{r(i)}\}} du \leq 
}
{\Cst(d)} \left(\frac{h}{r(i)}\right)^{d-1}
$$
Hence, if 
$r(i)\in  { (}0,h]$
\beq
\sum_{j:\,j\sim i} \vert \theta(i,j)\vert\le
\Cst(d) \frac{1}{\vert x-y\vert^{d-1}} \left( \frac{\vert x-y\vert}{ \vert i-x\vert}\right)^{d-1}= 
\frac{\Cst(d)}{\vert i-x\vert^{d-1}},
\eeq
where we used \eqref{deltah} and 
$\frac{h}{r (i)}=\frac{\vert u_{0}-x\vert}{\vert i-x\vert}\leq \frac{\vert x-y\vert}{\vert i-x\vert}.$
Similarly we have
if 
$r(i)\in [h,1 {)}$
\beq
\sum_{j:\,j\sim i} \vert \theta(i,j)\vert\le
 \frac{\Cst(d)}{\vert i-y\vert^{d-1}}.
\eeq

Now, under the condition $\overline\chi_{x,y}$, we know that if $i\in \hat R^x_{x,y}\subset R^x_{x,y}$, then 
  by Lemma \ref{estcond} 
$\gamma_{i,j}\ge c\vert i-x\vert^{-\beta}$ with $\beta=4\alpha<1$ and $c=b^{-4}/64$.
Similarly,  if $i\in \hat R^y_{x,y}\subset R^y_{x,y}$ then 
$\gamma_{i,j}\ge c\vert i-y\vert^{-\beta}.$ 
This implies
\beq
aD_{x,y}^N
&\le& 
\sum_{i\in \hat R^x_{x,y}, \; i\neq x} c^{-1}\vert i-x\vert^\beta  \left(\sum_{j:\,j\sim i} \vert \theta(i,j)\vert\right)^2+
\sum_{i\in \hat R^y_{x,y}, \; i\neq y} c^{-1}\vert i-y\vert^\beta  \left(\sum_{j:\,j\sim i} \vert \theta(i,j)\vert\right)^2
\\
&\le &
\Cst(d) \left(
\sum_{i\in \hat R^x_{x,y}, \; i\neq x} c^{-1}\vert i-x\vert^{\beta-2(d-1)} +
\sum_{i\in \hat R^y_{x,y}, \; i\neq y} c^{-1}\vert i-y\vert^{\beta-2(d-1)}\right)
\eeq
For $k\in \N$, 
$$
\left\vert\{i\in R_{x,y}, \;k\le \vert i-x\vert<k+1\}\right\vert \le \Cst(d) k^{d-1},
$$
and similarly for $x$ replaced by $y$.
Therefore
$$
aD_{x,y}^N\le \Cst(d) \sum_{k=1}^\infty  c^{-1} k^{\beta-(d-1)}\le \Cst(d,b).
$$
since $\beta=4\alpha\le1/2$ and $d\ge 3$.

\footnotesize
\bibliographystyle{plain}
\bibliography{transience}

\end{document}